\documentclass[11pt]{article}
\usepackage{amssymb}
\usepackage{amsfonts}
\usepackage{amsmath}

\newtheorem{theorem}{Theorem}

\newtheorem{corollary}[theorem]{Corollary}

\newtheorem{definition}[theorem]{Definition}

\newtheorem{lemma}[theorem]{Lemma}

\newtheorem{proposition}[theorem]{Proposition}
\newtheorem{remark}[theorem]{Remark}

\newenvironment{proof}[1][Proof]{\noindent\textbf{#1.} }{\ \rule{0.5em}{0.5em}}
\input{tcilatex}
\begin{document}

\title{On the Cauchy problem for integro-differential operators in Sobolev
classes and the martingale problem}
\author{R. Mikulevicius and H. Pragarauskas \\
University of Southern California, Los Angeles\\
Institute of Mathematics and Informatics, Vilnius }
\maketitle

\begin{abstract}
The existence and uniqueness in Sobolev spaces of solutions of the Cauchy
problem to parabolic integro-differential equation of the order $\alpha \in
(0,2)$ is investigated. The principal part of the operator has kernel $%
m(t,x,y)/|y|^{d+\alpha }$ with a bounded nondegenerate $m,$ H\"{o}lder in $x$
and measurable in $y.$ The lower order part has bounded and measurable
coefficients. The result is applied to prove the existence and uniqueness of
the corresponding martingale problem.

MSC classes: 45K05, 60J75, 35B65

\end{abstract}

\section{Introduction}

In this paper we consider the Cauchy problem%
\begin{eqnarray}
\partial _{t}u(t,x) &=&Lu(t,x)+f(t,x),(t,x)\in E=[0,T]\times \mathbf{R}^{d},
\label{intr1} \\
u(0,x) &=&0  \notag
\end{eqnarray}%
in fractional Sobolev spaces for a class of integrodifferential operators $%
L=A+B$ of the order $\alpha \in (0,2)$ whose principal part $A$ is of the
form%
\begin{eqnarray}
Av(t,x) &=&A_{t,x}v(x)=A_{t,z}v(x)|_{z=x},  \label{0} \\
A_{t,z}v(x) &=&\int \left[ v(x+y)-v(x)-\chi _{\alpha }(y)(\nabla v(x),y)%
\right] m(t,z,y)\frac{dy}{|y|^{d+\alpha }},  \notag
\end{eqnarray}%
$(t,z)\in E,x\in \mathbf{R}^{d},$ with $\chi _{\alpha }(y)=1_{\alpha
>1}+1_{\alpha =1}1_{\left\{ |y|\leq 1\right\} }$. We notice that the
operator $A$ is the generator of an $\alpha $-stable process. If $m=1,$then $%
A=c\left( -\Delta \right) ^{\alpha /2}$ (fractional Laplacian) is the
generator of a spherically symmetric $\alpha $-stable process. The part $B$
is a perturbing, subordinated operator.

In \cite{MiP922}, the problem was considered assuming that $m$ is Holder
continuous in $x$, homogeneous of order zero and smooth in $y$ and for some $%
\eta >0$ {\ 
\begin{equation}
\int_{S^{d-1}}|(w,\xi )|^{\alpha }m(t,x,w)\mu _{d-1}(dw)\geq \eta ,\quad
(t,x)\in E,|\xi |=1,  \label{1}
\end{equation}%
where }$\mu _{d-1}$ is the Lebesgue measure on the unit sphere $S^{d-1}$ in $%
\mathbf{R}^{d}$. In \cite{AbK09}, the existence and uniqueness of a solution
to (\ref{intr1}) in H\"{o}lder spaces was proved analytically for $m$ H\"{o}%
lder continuous in $x$, smooth in $y$ and such that 
\begin{equation}
K\geq m\geq \eta >0  \label{2}
\end{equation}%
without assumption of homogeneity in $y$. \ The elliptic problem $(L-\lambda
)u=f$ \ with $B=0$ and $m$ independent of $x$ in $\mathbf{R}^{d}$ was
considered in \cite{KimDong}$.$ The equation (\ref{intr1}) with $\alpha =1$
can be regarded as a linearization of the quasigeostrophic equation (see 
\cite{cav1}).

In this note, we consider he problem (\ref{intr1}), assuming that $m$ is
measurable, Holder continuous in $x$ and 
\begin{equation}
K\geq m\geq m_{0},  \label{3}
\end{equation}%
where the function $m_{0}=m_{0}(t,x,y)$ is smooth and homogeneous in $y$ and
satisfies (\ref{1}). So, the density $m$ can degenerate on a substantial set.

A certain aspect of the problem is that the symbol of the operator $A,$%
\begin{equation*}
\psi (t,x,\xi )=\int \left[ e^{i(\xi ,y)}-1-\chi _{\alpha }(y)i(\xi ,y)%
\right] m(t,x,y)\frac{dy}{|y|^{d+\alpha }}
\end{equation*}%
is not smooth in $\xi $ and the standard Fourier multiplier results (for
example, used in \cite{MiP922}) do not apply in this case. We start with
equation (\ref{intr1}) assuming that $B=0$, the input function $f$ is smooth
and the function $m=m(t,y)$ does not depend on $x$. In \cite{mikprag}, the
existence and uniqueness of a weak solution in Sobolev spaces was derived.
In this paper we show that the main part $A:H_{p}^{\alpha }\rightarrow L_{p}$
is bounded. Contrary to \cite{KimDong}, where H\"{o}lder estimates were
used, we give a direct proof based on the classical theory of singular
integrals (see Lemmas \ref{lem1}, \ref{lem2} below). The case of variable
coefficients is based on the a priori estimates using Sobolev embedding
theorem and the method in \cite{kry}.

As an application, we consider the martingale problem associated to $L$.
Since the lower part of $L$ has only measurable coefficients, we generalize
the results in \cite{mikprag1}.

The note is organized as follows. In Section 2, the main theorem is stated.
In Section 3, the essential technical results are presented. The main
theorem is proved in Section 4. In Section 5 we discuss the embedding of the
solution space. In Section 6 the existence and uniqueness of the associated
martingale problem is considered.

\section{Notation and main results}

Denote $E=[0,T]\times \mathbf{R}^{d}$, $\mathbf{N}=\{0,1,2,\ldots \}$, $%
\mathbf{R}_{0}^{d}=\mathbf{R}^{d}\backslash \{0\}$. If $x,y\in \mathbf{R}%
^{d} $, we write 
\begin{equation*}
(x,y)=\sum_{i=1}^{d}x_{i}y_{i},|x|=(x,x)^{1/2}.
\end{equation*}%
For a function $u=u(t,x)$ on $E$, we denote its partial derivatives by $%
\partial _{t}u=\partial u/\partial t,\partial _{i}u=\partial u/\partial
x_{i},\partial _{ij}^{2}u=\partial ^{2}u/\partial x_{i}\partial x_{j}$ and $%
D^{\gamma }u=\partial ^{|\gamma |}u/\partial x_{i}^{\gamma _{1}}\ldots
\partial x_{d}^{\gamma _{d}},$ where multiindex $\gamma =(\gamma _{1},\ldots
,\gamma _{d})\in \mathbf{N}^{d},\nabla u=(\partial _{1}u,\ldots ,\partial
_{d}u)$ denotes the gradient of $u$ with respect to $x$.

Let $L_{p}(E)$ be the space of $p$-integrable functions with norm%
\begin{equation*}
|f|_{p}=\left( \int_{0}^{T}\int |f(t,x)|^{p}dxdt\right) ^{1/p}.
\end{equation*}%
Similar space of functions on $\mathbf{R}^{d}$ is denoted $L_{p}(\mathbf{R}%
^{d}).$

Let $\mathcal{S}(\mathbf{R}^{d})$ be the Schwartz space of smooth
real-valued rapidly decreasing functions. We introduce the Sobolev space $%
H_{p}^{{\scriptsize \beta }}(\mathbf{R}^{d})$ of $f\in \mathcal{S}^{\prime }(%
\mathbf{R}^{d})$ with finite norm%
\begin{equation*}
|f|_{\beta ,p}=|\mathcal{F}^{-1}((1+|\xi |^{2})^{{\scriptsize \beta /2}}%
\mathcal{F}f)|_{p}.
\end{equation*}%
where $\mathcal{F}$ denotes the Fourier transform. We also introduce the
corresponding spaces of generalized functions on $E=[0,T]\times \mathbf{R}%
^{d}$: $H_{p}^{\beta }(E)$ consist of all measurable $S^{\prime }(\mathbf{R}%
^{d})$-valued functions $f$ on $[0,T]$ with finite norm 
\begin{equation*}
|f|_{\beta ,p}=\left\{ \int_{0}^{T}|f(t)|_{\beta ,p}^{p}dt\right\} ^{\frac{1%
}{p}}.
\end{equation*}

For $\alpha \in (0,2)$ and $u\in \mathcal{S}(\mathbf{R}^{d}),$ we define the
fractional Laplacian%
\begin{equation}
\partial ^{\alpha }u(x)=\int \nabla _{y}^{\alpha }u(x)\frac{dy}{%
|y|^{d+\alpha }},  \label{fo4}
\end{equation}%
where 
\begin{equation*}
\nabla _{y}^{\alpha }u(x)=u(x+y)-u(x)-\left( \nabla u(x),y\right) \chi
_{\alpha }(y)
\end{equation*}%
with $\chi ^{(\alpha )}(y)=\mathbf{1}_{\{|y|\leq 1\}}\mathbf{1}_{\{\alpha
=1\}}+\mathbf{1}_{\{\alpha \in (1,2)\}}$.

We denote $C_{b}^{\infty }(E)$ the space of bounded infinitely
differentiable in $x$ functions whose derivatives are bounded.

$C=C(\cdot ,\ldots ,\cdot )$ denotes constants depending only on quantities
appearing in parentheses. In a given context the same letter is (generally)
used to denote different constants depending on the same set of arguments.

Let $\alpha \in (0,2)$ be fixed. Let $m:E\times \mathbf{R}%
_{0}^{d}\rightarrow \lbrack 0,\infty ),b:E\rightarrow \mathbf{R}^{d}$ be
measurable functions. We also introduce an auxiliary function $%
m_{0}:[0,T]\times \mathbf{R}_{0}^{d}\rightarrow \lbrack 0,\infty )$ and fix
positive constants $K$ and $\eta $. Throughout the paper we assume that the
function $m_{0}$ satisfies the following conditions. \medskip

\noindent \textbf{Assumption} $\mathbf{A}_{0}.$ (i) The function $%
m_{0}=m_{0}(t,y)\geq 0$ is measurable, homogeneous in $y$ with index zero,
differentiable in $y$ up to the order $d_{0}=[\frac{d}{2}]+1$ and%
\begin{equation*}
|D_{y}^{\gamma }m_{0}^{(\alpha )}(t,y)|\leq K
\end{equation*}%
for all $t\in \lbrack 0,T]$, $y\in \mathbf{R}_{0}^{d}$ and multiindices $%
\gamma \in \mathbf{N}_{0}^{d}$ such that $|\gamma |\leq d_{0}$;

(ii) If $\alpha =1$, then for all $t\in \lbrack 0,T]$ 
\begin{equation*}
\int_{S^{d-1}}wm_{0}(t,w)\mu _{d-1}(dw)=0,
\end{equation*}%
where $S^{d-1}$ is the unit sphere in $\mathbf{R}^{d}$ and $\mu _{d-1}$ is
the Lebesgue measure on it;

(iii) For all $t\in \lbrack 0,T]$ 
\begin{equation*}
\inf_{|\xi |=1}\int_{S^{d-1}}|(w,\xi )|^{\alpha }m_{0}(t,w)\mu
_{d-1}(dw)\geq \eta >0.
\end{equation*}

\begin{remark}
\label{r10}\emph{The nondegenerateness assumption $A_{0}$ (iii) holds with
certain $\eta >0$ if, e.g.%
\begin{equation*}
\inf_{t\in \lbrack 0,T],w\in \Gamma }m_{0}(t,w)>0
\end{equation*}%
for a measurable subset $\Gamma \subset S^{d-1}$ of positive Lebesgue
measure. Therefore }$m_{0}$ \emph{can be zero on a substantial set.}
\end{remark}

Further we will use the following assumptions.

\textbf{Assumption }$\mathbf{A.}$ (i) For all $(t,x)\in E,y\in \mathbf{R}%
_{0}^{d}$,%
\begin{equation*}
K\geq m(t,x,y)\geq m_{0}(t,y),
\end{equation*}

where the function $m_{0}$ satisfies Assumption $\mathbf{A}_{0}$;

(ii) There is $\beta \in (0,1)$ and a continuous increasing function $%
w(\delta )$ such that 
\begin{equation*}
|m(t,x,y)-m(t,x^{\prime },y)|\leq w(|x-x^{\prime }|),t\in \lbrack
0,T],x,x^{\prime },y\in \mathbf{R}^{d},
\end{equation*}%
and%
\begin{equation*}
\int_{|y|\leq 1}w(|y|)\frac{dy}{|y|^{d+\beta }}<\infty ,\lim_{\delta
\rightarrow 0}w(\delta )\delta ^{-\beta }=0.
\end{equation*}

(iii) If $\alpha =1$, then for all $(t,x)\in E$ and $r\in (0,1),$%
\begin{equation*}
\int_{r<|y|\leq 1/r}ym(t,x,y)\frac{dy}{|y|^{d+\alpha }}=0.
\end{equation*}

We define the lower order operator $Bu(t,x)=B_{t,z}u(x)|_{z=x},(t,x)\in E,$
with

\begin{eqnarray*}
B_{t,z}u(x) &=&(b(t,z)\nabla u(x))1_{1<\alpha <2}+\int [u(x+y)-u(x) \\
&&-(\nabla u(x),y))1_{|y|\leq 1}1_{1<\alpha <2}]\pi (t,z,dy),
\end{eqnarray*}%
where $\left( \pi (t,z,dy)\right) $ is a measurable family of nonnegative
measures on $\mathbf{R}_{0}^{d}$ and $b(t,z)=\left( b^{i}(t,z)\right)
_{1\leq i\leq d}$ is a measurable function.

We will assume the following assumptions hold.

\textbf{Assumption B. }(i) For all $(t,x)\in E,$%
\begin{equation*}
|b(t,x)|+\int |\upsilon |^{\alpha }\wedge 1\pi (t,x,d\upsilon )\leq K;
\end{equation*}%
(ii) \textbf{\ }%
\begin{equation*}
\lim_{\varepsilon \rightarrow 0}\sup_{t,x}\int_{|\upsilon |\leq \varepsilon
}|\upsilon |^{\alpha }\pi (t,x,d\upsilon )=0;
\end{equation*}

(iii) For each $\varepsilon >0,$%
\begin{equation*}
\int_{0}^{T}\int \pi (t,x,\left\{ |v|>\varepsilon \right\} )dxdt<\infty .
\end{equation*}%
We write%
\begin{eqnarray}
Au(t,x) &=&A_{t}u(x)=A_{t,x}u(x),Bu(t,x)=B_{t}u(x)=B_{t,x}u(x),  \label{for2}
\\
Lu(t,x) &=&L_{t}u(x)=L_{t,x}u(x),L=A+B.  \notag
\end{eqnarray}%
According to Assumptions \textbf{A, B}, the operator $A$ represents the
principal part of $L$ and the operator $B$ is a lower order operator.

We consider the following Cauchy problem%
\begin{eqnarray}
\partial _{t}u(t,x) &=&(L-\lambda )u(t,x)+f(t,x),(t,x)\in H,  \label{eq1} \\
u(0,x) &=&0,x\in \mathbf{R}^{d},  \notag
\end{eqnarray}%
in Sobolev classes $H_{p}^{\alpha }(E)$, where $\lambda \geq 0$ and $f\in
L_{p}(E)$. More precisely, let $\mathcal{H}_{p}^{\alpha }\left( E\right) $
be the space of all functions $u\in H_{p}^{\alpha }(E)$ such that $u\left(
t,x\right) =\int_{0}^{t}F\left( s,x\right) \,ds,0\leq t\leq T,$ with $F\in
H_{p}^{\alpha }\left( E\right) .$ It is a Banach space with respect to the
norm 
\begin{equation*}
||u||_{\alpha ,p}=\left\vert u\right\vert _{\alpha ,p}+\left\vert
F\right\vert _{p}.
\end{equation*}

\begin{definition}
Let $f\in L_{p}(E).$ We say that $u\in \mathcal{H}_{p}^{\alpha }(E)$ is a
solution to \emph{(\ref{eq1})} if $Lu\in L_{p}(E)$ and 
\begin{equation}
{u}(t)=\int_{0}^{t}\bigl((L-\lambda )u(s)+f(s)\bigr)dt,0\leq t\leq T,
\label{defs}
\end{equation}%
in $L_{p}(\mathbf{R}^{d}).$
\end{definition}

If Assumptions \textbf{A} and \textbf{B} are satisfied, $p>\frac{d}{\alpha }%
\vee \frac{d}{\beta }\vee 2$, then $Lu\in L_{p}(E)$ (see Corollary \ref{cor4}
below and Lemma 7 in \cite{MiP922}). So, (\ref{defs}) is well defined.

The main result of the paper is the following theorem.

\begin{theorem}
\label{main}Let $\beta \in (0,1),p>\frac{d}{\beta },p\geq 2,$ and Assumption 
\textbf{A }be satisfied.

Then for any $f\in L_{p}(E)$ there exists a unique strong solution $u\in
H_{p}^{\alpha }(E)$ to (\ref{eq1}) with $B=0$. Moreover, there is a constant 
$N=N(T,\alpha ,\beta ,d,K,w,\eta )$ and a positive number $\lambda
_{1}=\lambda _{1}(T,\alpha ,\beta ,d,K,w,\eta )\geq 1$ such that%
\begin{eqnarray*}
|\partial _{t}u|_{p}+|u|_{\alpha ,p} &\leqslant &N|f|_{p}, \\
|u|_{p} &\leq &\frac{N}{\lambda }|f|_{p}\text{ if }\lambda \geq \lambda _{1}.
\end{eqnarray*}
\end{theorem}

We prove this theorem in Section 3 below.

In order to handle (\ref{eq1}) with the lower order part $Bu$, the following
estimate is needed.

\begin{lemma}
\label{lem0}(see Lemma 3.5 in \cite{MiP922}) Let $p>d/\alpha $. There is a
constant $N_{1}=N_{1}(p,\alpha ,d)$ such that%
\begin{equation*}
\left\vert \sup_{y\neq 0}\frac{|\nabla _{y}^{\alpha }v(\cdot )|}{|y|^{\alpha
}}\right\vert _{p}\leq N_{1}|\partial ^{\alpha }v|_{p},v\in C_{0}^{\infty }(%
\mathbf{R}^{d}).
\end{equation*}
\end{lemma}

Consider (\ref{eq1}) with $Bv(t,x)=B^{\varepsilon
_{0}}v(t,x)=B_{t,x}^{\varepsilon _{0}}v(x),(t,x)\in E,$ where 
\begin{eqnarray}
B_{t,z}^{\varepsilon _{0}}v(x) &=&(b(t,z)+1_{\alpha \in
(1.2)}\int_{\varepsilon _{0}\leq |y|\leq 1}y\pi (t,z,dy),\nabla v(x))
\label{f0} \\
&&+\int_{|y|\leq \varepsilon _{0}}\nabla _{y}^{\alpha }v(x)\pi
(t,z,dy),(t,z)\in E,x\in \mathbf{R}^{d},  \notag
\end{eqnarray}%
with some $\varepsilon _{0}\in (0,1].$

In the consideration of an associated martingale problem (see Section 5
below) the following statement is used.

\begin{theorem}
\label{main2}Let $\beta \in (0,1),p>\frac{d}{\beta }\vee \frac{d}{\alpha }%
\vee 2$ and Assumption \textbf{A }be satisfied. Let%
\begin{eqnarray*}
|b(t,x)|+\int_{\varepsilon _{0}\leq |y|\leq 1}|y|\pi (t,x,dy) &\leq &K, \\
\int_{|y|\leq \varepsilon _{0}}|y|^{\alpha }\pi (t,x,dy) &\leq &\delta
_{0},(t,x)\in E,
\end{eqnarray*}%
and $\delta _{0}NN_{1}\leq 1/2$, where $N$ is a constant of Theorem \ref%
{main}$.$ Then for any $f\in L_{p}(E)$ there exists a unique solution $u\in 
\mathcal{H}_{p}^{\alpha }(E)$ to (\ref{eq1}) with $B=B^{\varepsilon _{0}}$.
Moreover,%
\begin{eqnarray*}
|\partial _{t}u|_{p}+|u|_{\alpha ,p} &\leqslant &2N|f|_{p}, \\
|u|_{p} &\leq &\frac{2N}{\lambda }|f|_{p}\text{ if }\lambda \geq \lambda
_{1}.
\end{eqnarray*}
\end{theorem}

Finally, the results can be extended to

\begin{theorem}
\label{main3}Let $\beta \in (0,1),p>\frac{d}{\beta }\vee \frac{d}{\alpha }%
\vee 2$ and Assumptions \textbf{A, B }be satisfied.

Then for any $f\in L_{p}(E)$ there exists a unique solution $u\in \mathcal{H}%
_{p}^{\alpha }(E)$ to (\ref{eq1}). Moreover, there is a constant $N_{3}$
independent of $u$ such that%
\begin{equation*}
|\partial _{t}u|_{p}+|u|_{\alpha ,p}\leqslant N_{3}|f|_{p}.
\end{equation*}
\end{theorem}

\section{Auxiliary results}

In this section we present some auxiliary results.

\subsection{Continuity of the principal part}

First we prove the continuity of the operator $A$ in $L_{p}$-norm.

We will use the following equality for Sobolev norm estimates.

\begin{lemma}
\label{r1}$($Lemma 2.1 in \textup{\cite{Kom84}}$)$ For $\delta \in (0,1)$
and $u\in \mathcal{S}(\mathbf{R}^{d})$, 
\begin{equation}
u\left( x+y\right) -u(x)=C\int k^{(\delta )}(y,z)\partial ^{\delta }u(x-z)dz,
\label{22}
\end{equation}%
where the constant $C=C(\delta ,d)$ and 
\begin{equation*}
k^{(\delta )}(z,y)=|z+y|^{-d+\delta }-|z|^{-d+\delta }.
\end{equation*}%
Moreover, there is a constant $C=C(\delta ,d)$ such that for each $y\in 
\mathbf{R}^{d}$ 
\begin{equation*}
\int |k^{(\delta )}(z,y)|dz\leq C|y|^{\delta }.
\end{equation*}
\end{lemma}

For $\alpha \in (0,1)$ and a bounded measurable function $m(y)$, set for $%
u\in \mathcal{S}(\mathbf{R}^{d}),x\in \mathbf{R}^{d},$%
\begin{eqnarray*}
\mathcal{L}u(x) &=&\int [u(x+y)-u(x)]m(y)\frac{dy}{|y|^{d+\alpha }} \\
&=&\lim_{\varepsilon \rightarrow 0}\int [u(x+y)-u(x)]m_{\varepsilon }(y)%
\frac{dy}{|y|^{d+\alpha }} \\
&=&\lim_{\varepsilon \rightarrow 0}\int \int k^{(\alpha )}(z,y)\partial
^{\alpha }u(x-z)dzm_{\varepsilon }(y)\frac{dy}{|y|^{d+\alpha }} \\
&=&\lim_{\varepsilon \rightarrow 0}\int \int k^{(\alpha
)}(z,y)m_{\varepsilon }(y)\frac{dy}{|y|^{d+\alpha }}\partial ^{\alpha
}u(x-z)dz,
\end{eqnarray*}%
where%
\begin{eqnarray*}
m_{\varepsilon }(y) &=&\chi _{\left\{ \varepsilon \leq |y|\leq \varepsilon
^{-1}\right\} }m(y), \\
k^{(\alpha )}(x,y) &=&\frac{1}{|x+y|^{d-\alpha }}-\frac{1}{|x|^{d-\alpha }}
\end{eqnarray*}%
(see Lemma \ref{r1}).

For $\varepsilon \in (0,1),u,v\in \mathcal{S}(\mathbf{R}^{d}),$ consider%
\begin{eqnarray}
\mathcal{L}^{\varepsilon }u(x) &=&\mathcal{L}_{\alpha }^{\varepsilon
}u(x)=\int [u(x+y)-u(x)]m_{\varepsilon }(y)\frac{dy}{|y|^{d+\alpha }}
\label{for00} \\
&=&\int k_{\varepsilon }(z)\partial ^{\alpha }u(x-z)dz=\int k_{\varepsilon
}(x-z)\partial ^{\alpha }u(z)dz,  \notag
\end{eqnarray}%
and%
\begin{equation*}
\mathcal{K}^{\varepsilon }v(x)=\int k_{\varepsilon }(z)v(x-z)dz=\int
k_{\varepsilon }(x-z)v(z)dz
\end{equation*}%
where%
\begin{equation*}
k_{\varepsilon }(x)=\int k^{(\alpha )}(x,y)m_{\varepsilon }(y)\frac{dy}{%
|y|^{d+\alpha }},x\in \mathbf{R}^{d}.
\end{equation*}%
To prove the continuity $\mathcal{L}:H_{p}^{\alpha }(\mathbf{R}%
^{d})\rightarrow L_{p}(\mathbf{R}^{d})$, we will show that there is a
constant $C$ independent of $\varepsilon $,$v\in \mathcal{S}(\mathbf{R}^{d})$
such that%
\begin{equation}
|\mathcal{K}^{\varepsilon }v|_{p}\leq C|v|_{p}.  \label{fo1}
\end{equation}%
By \cite{stein}, (\ref{fo1}) will follow provided%
\begin{equation}
|\mathcal{K}^{\varepsilon }v|_{2}\leq C|v|_{2},  \label{fo2}
\end{equation}%
and%
\begin{equation}
\int_{|x|>4|s|}|k_{\varepsilon }(x-s)-k_{\varepsilon }(x)|dx\leq C\text{ for
all }s\in \mathbf{R}^{d}.  \label{fo20}
\end{equation}

\begin{remark}
\label{rem1}For any $t>0$, we have $k^{(\alpha )}(tx,ty)=t^{\alpha
-d}k^{(\alpha )}(x,y).$ Therefore%
\begin{equation*}
k_{\varepsilon }(tx)=t^{-d}\int k^{(\alpha )}(x,y)m_{\varepsilon }(ty)\frac{%
dy}{|y|^{d+\alpha }}=t^{-d}k_{\varepsilon }(t,x)
\end{equation*}%
with%
\begin{equation*}
k_{\varepsilon }(t,x)=\int k^{(\alpha )}(x,y)m_{\varepsilon }(ty)\frac{dy}{%
|y|^{d+\alpha }}.
\end{equation*}%
Note that for $x\neq 0,$%
\begin{eqnarray*}
k_{\varepsilon }(x) &=&k_{\varepsilon }(|x|\hat{x})=|x|^{-d}\int k^{(\alpha
)}(\hat{x},y)m_{\varepsilon }(|x|y)\frac{dy}{|y|^{d+\alpha }} \\
&=&|x|^{-d}k_{\varepsilon }(|x|,\hat{x}),
\end{eqnarray*}%
where $\hat{x}=x/|x|$.
\end{remark}

\begin{lemma}
\label{lem1}Let $\alpha \in (0,1),|m(y)|\leq 1,y\in \mathbf{R}^{d}$. Then
for each $p>1$ there is a constant $C$ independent of $u$ and $\varepsilon $
such that,%
\begin{equation*}
|\mathcal{K}^{\varepsilon }u|_{p}\leq C|u|_{p},u\in L_{p}(\mathbf{R}^{d}).
\end{equation*}
\end{lemma}

\begin{proof}
It is enough to show that (\ref{fo2}) and (\ref{fo20}) hold. By Lemma 1 of
Chapter 5.1 in \cite{stein}, it follows 
\begin{eqnarray*}
\hat{k}_{\varepsilon }(\xi ) &=&\int e^{-i(x,\xi )}k_{\varepsilon
}(x)dx=C|\xi |^{-\alpha }\int [e^{-i(\xi ,y)}-1]m_{\varepsilon }(y)\frac{dy}{%
|y|^{d+\alpha }} \\
&=&\int [e^{-i(\hat{\xi},y)}-1]m_{\varepsilon }(y/|\xi |)\frac{dy}{%
|y|^{d+\alpha }},
\end{eqnarray*}%
where $\hat{\xi}=\xi /|\xi |$. Therefore, by Parseval's equality, (\ref{fo2}%
) holds for $v\in \mathcal{S}(\mathbf{R}^{d})$. The key estimate is (\ref%
{fo20}). By Remark \ref{rem1}, denoting $\hat{s}=s/|s|,$ we have%
\begin{eqnarray*}
\int_{|x|>4|s|}|k_{\varepsilon }(x-s)-k_{\varepsilon }(x)|dx
&=&\int_{|x|/|s|>4}|k_{\varepsilon }(|s|(\frac{x}{|s|}-\hat{s}%
)-k_{\varepsilon }(|s|\frac{x}{|s|})|dx \\
&=&|s|^{d}\int_{|x|>4}|k_{\varepsilon }(|s|(x-\hat{s}))-k_{\varepsilon
}(|s|x)|dx \\
&=&\int_{|x|>4}|k_{\varepsilon }(|s|,x-\hat{s})-k_{\varepsilon }(|s|,x)|dx
\end{eqnarray*}%
and it is enough to prove that%
\begin{equation}
\int_{|x|>4}|k_{\varepsilon }(|s|,x-\hat{s})-k_{\varepsilon }(|s|,x)|dx\leq M%
\text{ for all }s\in \mathbf{R}^{d},\hat{s}=s/|s|.  \label{for3}
\end{equation}%
We will estimate for $|x|\geq 4,s\in \mathbf{R}^{d},\hat{s}=s/|s|,$ the
difference%
\begin{eqnarray*}
&&|k(|s|,x-\hat{s})-k(|s|,x)| \\
&=&\int [k^{(\alpha )}(x-\hat{s},y)-k^{(\alpha )}(x,y)]m_{\varepsilon }(|s|y)%
\frac{dy}{|y|^{d+\alpha }} \\
&=&\int_{|y|\leq |x|/2}...+\int_{|y|>|x|/2}...=A_{1}+A_{2}.
\end{eqnarray*}

Let%
\begin{equation*}
F(t)=\frac{1}{|x-t\hat{s}+y|^{d-\alpha .}}-\frac{1}{|x-t\hat{s}|^{d-\alpha }}%
,0\leq t\leq 1.
\end{equation*}%
If a segment connecting $x$ and $x-\hat{s}$ does not contain zero, then%
\begin{eqnarray}
&&|k^{(\alpha )}(x-\hat{s},y)-k^{(\alpha )}(x,y)|  \label{for5} \\
&=&|F(1)-F(0)|\leq \int_{0}^{1}|F^{\prime }(t)|dt  \notag
\end{eqnarray}%
with 
\begin{equation*}
F^{\prime }(t)=\left( \alpha -d\right) [-\frac{1}{|x-t\hat{s}+y|^{d-\alpha
+1}}\frac{\left( x-t\hat{s}+y,\hat{s}\right) }{|x-t\hat{s}+y|}+\frac{1}{|x-t%
\hat{s}|^{d-\alpha +1}}\frac{\left( x-t\hat{s},\hat{s}\right) }{|x-t\hat{s}|}%
],
\end{equation*}%
and%
\begin{eqnarray}
|F^{\prime }(t)| &\leq &C\left\vert \frac{1}{|x-t\hat{s}+y|^{d-\alpha +1}}-%
\frac{1}{|x-t\hat{s}|^{d-\alpha +1}}\right\vert  \label{for4} \\
&&+\frac{1}{|x-t\hat{s}|^{d-\alpha +1}}(1+\frac{1}{|x-t\hat{s}|})\left(
|y|\wedge 1\right) .  \notag
\end{eqnarray}

\emph{Estimate of }$A_{1}$. Let $|x|\geq 4,z=x-t\hat{s},t\in \lbrack 0,1],$
and $|y|\leq |x|/2$. In this case, $x+y|\geq |x|-|y|\geq |x|/2\geq 2,$ 
\begin{eqnarray*}
C|x| &\geq &|z+y|\geq x|/4\geq 1, \\
C|x| &\geq &|z|\geq |x|-1\geq 3|x|/4\geq 3
\end{eqnarray*}%
and (\ref{for5}) holds. Since%
\begin{eqnarray*}
&&\int_{|y|\leq |x|/2}[\left\vert \frac{1}{|z+y|^{d-\alpha +1}}-\frac{1}{%
|z|^{d-\alpha +1}}\right\vert \frac{dy}{|y|^{d+\alpha }} \\
&\leq &\frac{1}{|z|^{d+1}}\int_{|y|\leq 2/3}\left\vert \frac{1}{|\hat{z}%
+y|^{d-\alpha +1}}-1\right\vert \frac{dy}{|y|^{d+\alpha }} \\
&\leq &\frac{C}{|x|^{d+1}},
\end{eqnarray*}%
and%
\begin{equation*}
\int_{|y|\leq |x|/2}(|y|\wedge 1)]\frac{dy}{|y|^{d+\alpha }}\leq C,
\end{equation*}%
It follows by (\ref{for4}), that 
\begin{eqnarray*}
|A_{1}| &\leq &C\int_{|y|\leq |x|/2}[\left\vert \frac{1}{|z+y|^{d-\alpha +1}}%
-\frac{1}{|z|^{d-\alpha +1}}\right\vert +\frac{1}{|z|^{d-\alpha +1}}%
(|y|\wedge 1)]\frac{dy}{|y|^{d+\alpha }} \\
&\leq &C[\frac{1}{|x|^{d+1}}+\frac{1}{|x|^{d-\alpha +1}}].
\end{eqnarray*}

\emph{Estimate of }$A_{2}.$ Let$|x|\geq 4,|y|>|x|/2$. In this case we split 
\begin{eqnarray*}
A_{2} &=&\int_{|y|>|x|/2}[k^{(\alpha )}(x-\hat{s},y)-k^{(\alpha
)}(x,y)]m(|s|y)\frac{dy}{|y|^{d+\alpha }} \\
&=&\int_{\left\{ |x|-3/2\geq |y|>|x|/2\right\} \cup \left\{
|y|>|x|+3/2\right\} }...+\int_{\left\{ |x|-3/2\leq |y|\leq |x|+3/2\right\}
}... \\
&=&B_{1}+B_{2}.
\end{eqnarray*}

If $|x|-3/2\geq |y|>|x|/2$ or $|y|>|x|+3/2,$ then we can apply (\ref{for5})
and \ref{for4}). For $z=x-t\hat{s}$ we have $|z+y|\geq \frac{1}{2}$ and 
\begin{equation*}
|F^{\prime }(t)|\leq C[\frac{1}{|z+y|^{d-\alpha +1}}+\frac{1}{|z|^{d-\alpha
+1}}].
\end{equation*}%
Therefore 
\begin{eqnarray*}
|B_{1}| &\leq &C[\int_{\{\frac{|x|}{2}\leq |y|\}}\frac{1}{|z|^{d-\alpha +1}}%
\frac{dy}{|y|^{d+\alpha }}+\int_{\{\frac{|x|}{2}\leq |y|\leq |x|-\frac{3}{2}%
\}}\frac{1}{|z+y|^{d-\alpha +1}}\frac{dy}{|y|^{d+\alpha }} \\
&&+\int_{\{|x|+\frac{3}{2}\leq |y|\}}\frac{1}{|z+y|^{d-\alpha +1}}\frac{dy}{%
|y|^{d+\alpha }}] \\
&=&B_{11}+B_{12}+B_{13}.
\end{eqnarray*}%
Now%
\begin{equation*}
B_{11}=C\int_{\{\frac{|x|}{2}\leq |y|\}}\frac{1}{|z|^{d-\alpha +1}}\frac{dy}{%
|y|^{d+\alpha }}\leq \frac{C}{|x|^{d+1}},
\end{equation*}%
and%
\begin{eqnarray*}
B_{12} &\leq &C|z|^{-d-1}\int_{\{\frac{|x|}{2|z|}\leq |y|\leq (|x|-\frac{3}{2%
})/|z|\}}\frac{1}{|\hat{z}+y|^{d-\alpha +1}}\frac{dy}{|y|^{d+\alpha }}, \\
B_{13} &\leq &C|z|^{-d-1}\int_{\{\frac{|x|}{|z|}+\frac{3}{2|z|}\leq |y|\}}%
\frac{1}{|\hat{z}+y|^{d-\alpha +1}}\frac{dy}{|y|^{d+\alpha }}
\end{eqnarray*}%
with $\hat{z}=z/|z|.$ If $\frac{|x|}{2|z|}\leq |y|\leq (|x|-\frac{3}{2})/|z|$
or $\frac{|x|}{|z|}+\frac{3}{2|z|}\leq |y|$, then $|\hat{z}+y|\geq \frac{1}{%
2|z|}$ and $|y|\geq 1/3$. Therefore%
\begin{eqnarray*}
B_{12} &\leq &C|z|^{-d-1}\int_{\{|\hat{z}+y|\geq \frac{1}{2|z|}\}}\frac{dy}{|%
\hat{z}+y|^{d-\alpha +1}} \\
&\leq &C|z|^{-d-\alpha }\leq C|x|^{-d-\alpha }
\end{eqnarray*}%
and 
\begin{eqnarray*}
&&B_{13}\leq C|z|^{-d-1}\int_{|\hat{z}+y|\geq \frac{1}{2|z|}}\frac{1}{|\hat{z%
}+y|^{d-\alpha +1}}dy \\
&=&C|z|^{-d-\alpha }\leq C|x|^{-d-\alpha }.
\end{eqnarray*}

Now we estimate $B_{2}$. If $|x|-\frac{3}{2}\leq |y|\leq |x|+\frac{3}{2}$,
then we estimate directly. First we have 
\begin{eqnarray*}
&&\int_{\left\{ |x|-\frac{3}{2}\leq |y|\leq |x|+\frac{3}{2}\right\} }\frac{1%
}{|z|^{d-\alpha }}\frac{dy}{|y|^{d+\alpha }} \\
&\leq &C|x|^{\alpha -d}\left\vert \frac{1}{(|x|-\frac{3}{2})^{\alpha }}-%
\frac{1}{(|x|+\frac{3}{2})^{\alpha }}\right\vert \\
&\leq &C|x|^{-d-\alpha }.
\end{eqnarray*}%
Then, for $z=x-t\hat{s}$ with $t\in \lbrack 0,1],$ we have $\frac{2}{3}\leq
1-\frac{1}{|z|}\leq \frac{|x|}{|z|}\leq 1+\frac{1}{|z|}\leq \frac{4}{3}$ and%
\begin{eqnarray*}
&&\int_{\{|x|-\frac{3}{2}\leq |y|\leq |x|+\frac{3}{2}\}}\frac{1}{%
|z+y|^{d-\alpha }}\frac{dy}{|y|^{d+\alpha }} \\
&\leq &|z|^{-d}\int_{\left\{ 1-\frac{5}{2|z|}\leq |y|\leq 1+\frac{5}{2|z|}%
\right\} }\frac{1}{|\hat{z}+y|^{d-\alpha }}\frac{dy}{|y|^{d+\alpha }} \\
&\leq &|z|^{-d}\int_{\left\{ 1-\frac{5}{2|z|}\leq |y|\leq 1+\frac{5}{2|z|},|%
\hat{z}+y|>|z|^{-\alpha /d}\right\} }...+\int_{\left\{ 1-\frac{5}{2|z|}\leq
|y|\leq 1+\frac{5}{2|z|},|\hat{z}+y|\leq |z|^{-\alpha /d}\right\} }... \\
&\leq &C[|z|^{-d}|z|^{\frac{\alpha }{d}(d-\alpha )}\int_{\left\{ 1-\frac{5}{%
2|z|}\leq |y|\leq 1+\frac{5}{2|z|}\right\} }\frac{dy}{|y|^{d+\alpha }} \\
&&+|z|^{-d}\int_{|\hat{z}+y|\leq |z|^{-\alpha /d}}\frac{dy}{|\hat{z}%
+y|^{d-\alpha }}]\leq C|z|^{-\alpha ^{2}/d-d}\leq C|x|^{-\alpha ^{2}/d-d}
\end{eqnarray*}%
with $\hat{z}=z/|z|$. Therefore, 
\begin{equation*}
|B_{2}|\leq C[|x|^{-\alpha ^{2}/d-d}+|x|^{-d-1}].
\end{equation*}

The statement follows.
\end{proof}

For a bounded measurable $m(y),y\in \mathbf{R}^{d},$ and $\alpha \in (0,2)$,
set for $u\in \mathcal{S}(\mathbf{R}^{d}),x\in \mathbf{R}^{d},$%
\begin{equation*}
\mathcal{L}u(x)=\mathcal{L}_{\alpha }u(x)=\int \nabla ^{\alpha }u(x)m(y)%
\frac{dy}{|y|^{d+\alpha }},
\end{equation*}%
where 
\begin{equation*}
\nabla ^{\alpha }u(x)=u(x+y)-u(x)-\chi _{\alpha }(y)\left( \nabla
u(x),y\right)
\end{equation*}%
with $\chi _{\alpha }(y)=1_{\alpha \in (1,2)}+1_{\alpha =1}1_{|y|\leq 1}.$

\begin{lemma}
\label{lem2}Let $|m(y)|\leq K,y\in \mathbf{R}^{d},p>1,$ and $\alpha \in
(0,2).$ Assume%
\begin{equation*}
\int_{r\leq |y|\leq R}ym(y)\frac{dy}{|y|^{d+\alpha }}=0
\end{equation*}%
for any $0<r<R$ if $\alpha =1$. Then there is a constant $C$ such that%
\begin{equation*}
|\mathcal{L}_{\alpha }u|_{p}\leq CK|\partial ^{\alpha }u|_{p},u\in L_{p}(%
\mathbf{R}^{d}).
\end{equation*}
\end{lemma}

\begin{proof}
If $\alpha \in (0,1)$, then for $u\in \mathcal{S}(\mathbf{R}^{d})$ we have 
\begin{equation*}
\mathcal{L}u(x)=\lim_{\varepsilon \rightarrow 0}\mathcal{L}^{\varepsilon
}u(x),x\in \mathbf{R}^{d},
\end{equation*}%
and by Lemma \ref{lem1} there is a constant $C$ independent on $u$ such that%
\begin{equation*}
|K^{-1}\mathcal{L}u|_{p}\leq C|\partial ^{\alpha }u|_{p}
\end{equation*}%
or 
\begin{equation*}
|\mathcal{L}_{\alpha }u|_{p}\leq CK|\partial ^{\alpha }u|_{p},u\in \mathcal{S%
}(\mathbf{R}^{d}).
\end{equation*}

If $\alpha \in (1,2)$, then it follows by Lemma \ref{r1} that for $u\in 
\mathcal{S}(\mathbf{R}^{d}),$ 
\begin{eqnarray*}
\mathcal{L}_{\alpha }u(x) &=&\int \left[ u(x+y)-u(x)-(\nabla u(x),y)\right]
m(y)\frac{dy}{|y|^{d+\alpha }} \\
&=&\int \left[ \int_{0}^{1}\left( \nabla u(x+sy)-\nabla u(x),y\right) \right]
m(y)\frac{dsdy}{|y|^{d+\alpha }} \\
&=&\int \left[ \left( \nabla u(x+y)-\nabla u(x),\frac{y}{|y|}\right) \right]
M(y)\frac{dy}{|y|^{d+\alpha -1}}
\end{eqnarray*}%
with%
\begin{equation*}
M(y)=\int_{0}^{1}m(y/s)s^{-1+\alpha }ds,y\in \mathbf{R}^{d}\text{.}
\end{equation*}%
Therefore, the estimate reduces to the case of $\alpha \in (0,1)$: there is
a constant $C$ independent of $u\in \mathcal{S}(\mathbf{R}^{d})$ such that%
\begin{equation*}
|\mathcal{L}_{\alpha }u|_{p}\leq CK|\partial ^{\alpha -1}\nabla u|_{p}\leq
C|\partial ^{\alpha }u|_{p}.
\end{equation*}

If $\alpha =1$, then for $u\in \mathcal{S}(\mathbf{R}^{d}),$ 
\begin{eqnarray*}
\mathcal{L}_{1}u(x) &=&\int [u(x+y)-u(x)-(\nabla u(x),y)1_{|y|\leq 1}]m(y)%
\frac{dy}{|y|^{d+1}} \\
&=&\lim_{\varepsilon \rightarrow 0}\int [u(x+y)-u(x)]m_{\varepsilon }(y)%
\frac{dy}{|y|^{d+1}}
\end{eqnarray*}%
with $m_{\varepsilon }(y)=m(y)1_{\varepsilon ^{-1}\geq |y|>\varepsilon
},y\in \mathbf{R}^{d}$. By Lemma \ref{r1}, for $u\in \mathcal{S}(\mathbf{R}%
^{d}),x\in \mathbf{R}^{d},$ 
\begin{eqnarray*}
&&\int [u(x+y)-u(x)]m_{\varepsilon }(y)\frac{dy}{|y|^{d+1}} \\
&=&\int \int k^{(1/2)}(z,y)\partial ^{1/2}u(x-z)dzm_{\varepsilon }(y)\frac{dy%
}{|y|^{d+1}} \\
&=&\int \int k^{(1/2)}(z,y)[\partial ^{1/2}u(x-z)-\partial
^{1/2}u(x)]dzm_{\varepsilon }(y)\frac{dy}{|y|^{d+1}}
\end{eqnarray*}%
and%
\begin{eqnarray*}
\mathcal{L}_{1}u(x) &=&\lim_{\varepsilon \rightarrow 0}\int
[u(x+y)-u(x)]m_{\varepsilon }(y)\frac{dy}{|y|^{d+1}} \\
&=&\lim_{\varepsilon \rightarrow 0}\int \int k^{(1/2)}(z,y)m_{\varepsilon
}(y)\frac{dy}{|y|^{d+1}}[\partial ^{1/2}u(x-z)-\partial ^{1/2}u(x)]dz.
\end{eqnarray*}%
Obviously,%
\begin{eqnarray*}
&&\int k^{(1/2)}(z,y)m_{\varepsilon }(y)\frac{dy}{|y|^{d+1}} \\
&=&\frac{1}{|z|^{d+\frac{1}{2}}}\int \left( \frac{1}{|\hat{z}+y|^{d-\frac{1}{%
2}}}-1\right) m_{\varepsilon }(|z|y)\frac{dy}{|y|^{d+1}} \\
&=&\frac{1}{|z|^{d+\frac{1}{2}}}M_{\varepsilon }(z),
\end{eqnarray*}%
where $\hat{z}=z/|z|$. Since for $\varepsilon \in (0,1/2)$, we have $%
|M_{\varepsilon }(z)|\leq CK$ and for $z\in \mathbf{R}^{d},$%
\begin{eqnarray*}
\lim_{\varepsilon \rightarrow 0}M_{\varepsilon }(z) &=&M(z) \\
&=&\int_{\{|y|\leq \frac{1}{2}\}}\left( \frac{1}{|\hat{z}+y|^{d-\frac{1}{2}}}%
-1+(d-\frac{1}{2})(\hat{z},y)\right) m(|z|y)\frac{dy}{|y|^{d+1}} \\
&&+\int_{\{|y|>\frac{1}{2}\}}\left( \frac{1}{|\hat{z}+y|^{d-\frac{1}{2}}}%
-1\right) m(|z|y)\frac{dy}{|y|^{d+1}},
\end{eqnarray*}%
it follows that%
\begin{eqnarray*}
\mathcal{L}^{1}u(x) &=&\int [u(x+y)-u(x)-(\nabla u(x),y)1_{|y|\leq 1}]m(y)%
\frac{dy}{|y|^{d+1}} \\
&=&\int [\partial ^{1/2}u(x+z)-\partial ^{1/2}u(x)]M(-z)\frac{dz}{|z|^{d+%
\frac{1}{2}}}
\end{eqnarray*}%
with $|M(z)|\leq CK,z\in \mathbf{R}^{d}$ and the estimate follows from the
case $\alpha =1/2$.
\end{proof}

Now we investigate the continuity of the main part $A$ with $m$ depending on
the spacial variable. For a bounded measurable $m(x,y),x,y\in \mathbf{R}^{d}$%
, consider the operator $\mathcal{A}u(x)=\mathcal{A}_{z}u(x)|_{z=x},x\in 
\mathbf{R}^{d},$ with $u\in \mathcal{S}(\mathbf{R}^{d})$ and%
\begin{equation}
\mathcal{A}_{z}u(x)=\mathcal{A}_{z}u(x)=\mathcal{A}_{z}^{m}u(x)=\int \nabla
_{y}^{\alpha }u(x)m(z,y)\frac{dy}{|y|^{d+\alpha }},z,x\in \mathbf{R}^{d}.
\label{f1}
\end{equation}

\begin{lemma}
\label{rem5}Assume $\beta \in (0,1),p>d/\beta $. Let for each $y\in \mathbf{R%
}^{d},$ $m(\cdot ,y)\in H_{p}^{\beta }(\mathbf{R}^{d})$ and%
\begin{equation*}
|m(z,y)|+|\partial _{z}^{\beta }m(z,y)|<\infty .
\end{equation*}%
Then%
\begin{equation*}
|\mathcal{A}u|_{p}^{p}\leq C|\partial ^{\alpha }u|_{p}^{p}\int
\sup_{y}[|m(z,y)|^{p}+|\partial ^{\beta }m(z,y)|^{p}]dz.
\end{equation*}
\end{lemma}

\begin{proof}
By Sobolev embedding theorem, there is a constant $C$ such that%
\begin{eqnarray*}
|\mathcal{A}_{x}u(x)|^{p} &\leq &\sup_{z}|\mathcal{A}_{z}u(x)|^{p}\leq C\int
[|\mathcal{A}_{z}u(x)|^{p}+|\partial _{z}^{\beta }\mathcal{A}_{z}u(x)|^{p}]dz
\\
&=&C\int [|\mathcal{A}_{z}^{m}u(x)|^{p}+|\mathcal{A}_{z}^{\partial
_{z}^{\beta }m}u(x)|^{p}]dz,x\in \mathbf{R}^{d},
\end{eqnarray*}%
and by Lemma \ref{lem2}%
\begin{eqnarray*}
|\mathcal{A}u|_{p}^{p} &\leq &C\int [|\mathcal{A}_{z}u|_{p}^{p}+|\mathcal{A}%
_{z}^{\partial _{z}^{\beta }m}u|_{p}^{p}]dz \\
&\leq &C|\partial ^{\alpha }u|_{p}^{p}\int \sup_{y}[|m(z,y)|^{p}+|\partial
^{\beta }m(z,y)|^{p}]dz.
\end{eqnarray*}
\end{proof}

The following statement holds.

\begin{lemma}
\label{rem6}Let $u\in C_{0}^{\infty }(\mathbf{R}^{d})$ have its support in a
unit ball$.$ Assume $\beta \in (0,1),p>d/\beta $, and 
\begin{equation*}
|m(z,y)|+|\partial _{z}^{\beta }m(z,y)|\leq K,z,y\in \mathbf{R}^{d}.
\end{equation*}%
Then there is a constant $C=C(\alpha ,p,\beta ,d)$ independent of $u$ such
that%
\begin{equation*}
|\mathcal{A}u|_{p}^{p}\leq CK^{p}|u|_{\alpha ,p}^{p}.
\end{equation*}
\end{lemma}

\begin{proof}
Let the support of $u$ is a subset of the ball centered at $x_{0}$ with
radius $1$. Then for $x\in \mathbf{R}^{d},$%
\begin{eqnarray*}
\mathcal{A}u(x) &=&\int_{|y|\leq 1}\nabla _{y}^{\alpha }u(x)m(x,y)\frac{dy}{%
|y|^{d+\alpha }}+\int_{|y|>1}\nabla _{y}^{\alpha }u(x)m(x,y)\frac{dy}{%
|y|^{d+\alpha }} \\
&=&A_{1}(x)+A_{2}(x),
\end{eqnarray*}%
and%
\begin{equation*}
|A_{2}|_{p}^{p}\leq C|u|_{\alpha ,p}^{p}\sup_{z,y}|m(z,y)|^{p}.
\end{equation*}%
Let $\varphi \in C_{0}^{\infty }(\mathbf{R}^{d}),0\leq \varphi \leq
1,\varphi (x)=1$ if $|x|\leq 1,$ and $\varphi (x)=0$ if $|x|>2$. Then%
\begin{equation*}
A_{1}(x)=\int \nabla _{y}^{\alpha }u(x)\varphi \left( \frac{x-x_{0}}{2}%
\right) m(x,y)1_{|y|\leq 1}\frac{dy}{|y|^{d+\alpha }}
\end{equation*}%
and by Lemma \ref{rem5},%
\begin{equation*}
|A_{1}|_{p}^{p}\leq C|\partial ^{\alpha }u|_{p}^{p}\{\int \sup_{y}[|\varphi
\left( \frac{z-x_{0}}{2}\right) m(z,y)|^{p}+\left\vert \partial _{z}^{\beta
}\left( \varphi \left( \frac{z-x_{0}}{2}\right) m(z,y)\right) \right\vert
^{p}]dz.
\end{equation*}%
For each $y,z\in \mathbf{R}^{d},$%
\begin{eqnarray*}
&&\partial _{z}^{\beta }\left( \varphi \left( \frac{z-x_{0}}{2}\right)
m(z,y)\right) \\
&=&\int_{|\upsilon |>1}\left[ \varphi \left( \frac{z+\upsilon -x_{0}}{2}%
\right) m(z+\upsilon ,y)-\varphi \left( \frac{z-x_{0}}{2}\right) m(z,y)%
\right] \frac{d\upsilon }{|v|^{d+\beta }} \\
&&+\varphi \left( \frac{z-x_{0}}{2}\right) \int_{|\upsilon |\leq 1}\left[
m(z+\upsilon ,y)-m(z,y)\right] \frac{d\upsilon }{|v|^{d+\beta }} \\
&&+m(z,y)\int_{|\upsilon |\leq 1}\left[ \varphi \left( \frac{z+\upsilon
-x_{0}}{2}\right) -\varphi \left( \frac{z-x_{0}}{2}\right) \right] \frac{%
d\upsilon }{|v|^{d+\beta }} \\
&&+\int_{|\upsilon |\leq 1}\left[ m(z+\upsilon ,y)-m(z,y)\right] \left[
\varphi \left( \frac{z+\upsilon -x_{0}}{2}\right) -\varphi \left( \frac{%
z-x_{0}}{2}\right) \right] \frac{d\upsilon }{|\upsilon |^{d+\beta }} \\
&=&G_{0}(z,y)+G_{1}(z,y)+G_{2}(z,y)+G_{3}(z,y).
\end{eqnarray*}%
Obviously,%
\begin{equation*}
\int \sup_{y}|G_{0}(z,y)|^{p}dz\leq CK^{p}\int |\varphi (z)|^{p}dz.
\end{equation*}%
Since 
\begin{eqnarray*}
&&\int_{|\upsilon |\leq 1}\left[ m(z+\upsilon ,y)-m(z,y)\right] \frac{%
d\upsilon }{|v|^{d+\beta }} \\
&=&\partial _{z}^{\beta }m(z,y)-\int_{|\upsilon |>1}\left[ m(z+\upsilon
,y)-m(z,y)\right] \frac{d\upsilon }{|v|^{d+\beta }},
\end{eqnarray*}%
we have%
\begin{equation*}
\int \sup_{y}|G_{1}(z,y)|^{p}dz\leq CK^{p}\int |\varphi (z)|^{p}dz.
\end{equation*}%
Also,%
\begin{eqnarray*}
&&\int_{|\upsilon |\leq 1}\left\vert \varphi \left( \frac{z+\upsilon -x_{0}}{%
2}\right) -\varphi \left( \frac{z-x_{0}}{2}\right) \right\vert \frac{%
d\upsilon }{|v|^{d+\beta }} \\
&\leq &C\int_{0}^{1}\int_{|\upsilon |\leq 1}\left\vert \nabla \varphi \left( 
\frac{z+s\upsilon -x_{0}}{2}\right) \right\vert \frac{d\upsilon }{|\upsilon
|^{d+\beta -1}}
\end{eqnarray*}%
\newline
and%
\begin{equation*}
\int \sup_{y}|G_{2}(z,y)|^{p}dz\leq CK^{p}|\nabla \varphi |_{p}^{p}.
\end{equation*}%
Finally,%
\begin{equation*}
\int \sup_{y}|G_{3}(z,y)|^{p}dz\leq C\sup_{y,z}|m(z,y)|^{p}\left\vert \nabla
\varphi \right\vert _{p}^{p}
\end{equation*}%
and the statement follows.
\end{proof}

\begin{corollary}
\label{cor4}Assume $\beta \in (0,1),p>d/\beta $, and 
\begin{equation*}
|m(z,y)|+|\partial _{z}^{\beta }m(z,y)|\leq K,z,y\in \mathbf{R}^{d}.
\end{equation*}%
Then there is a constant $C=C(\alpha ,\beta ,p,d)$ such that%
\begin{equation*}
|\mathcal{A}u|_{p}^{p}\leq CK^{p}|u|_{\alpha ,p}^{p}.
\end{equation*}
\end{corollary}

\begin{proof}
Let $\zeta \in C_{0}^{\infty }(\mathbf{R}^{d})$ be such that $\int |\zeta
|^{p}dx=1$ and $\zeta $ has its support in the unit ball centered at the
origin. Then for each $x\in \mathbf{R}^{d},$%
\begin{equation*}
|\mathcal{A}u(x)|^{p}=\int |\zeta (x-\upsilon )\mathcal{A}u(x)|^{p}d\upsilon
.
\end{equation*}%
Obviously,%
\begin{eqnarray}
&&\zeta (x-\upsilon )\mathcal{A}u(x)  \label{fo6} \\
&=&\mathcal{A}\left( u(\cdot )\zeta (\cdot -\upsilon )\right) -u(x)\mathcal{A%
}\zeta (x-\upsilon )  \notag \\
&&+\int \left( u(x+y)-u(x)\right) \left( \zeta (x+y-\upsilon )-\zeta
(x-\upsilon )\right) m(x,y)\frac{dy}{|y|^{d+\alpha }}.  \notag
\end{eqnarray}%
Now%
\begin{eqnarray*}
\mathcal{A}\zeta (x-\upsilon ) &\mathcal{=}&\int \left[ \zeta (x+y-\upsilon
)-\zeta (x-\upsilon )\right] m(x,y)\frac{dy}{|y|^{d+\alpha }} \\
&=&\int_{|y|>1}\left[ \zeta (x+y-\upsilon )-\zeta (x-\upsilon )\right] m(x,y)%
\frac{dy}{|y|^{d+\alpha }} \\
&&+\int_{|y|\leq 1}\left[ \zeta (x+y-\upsilon )-\zeta (x-\upsilon )\right]
m(x,y)\frac{dy}{|y|^{d+\alpha }} \\
&=&A_{1}(x,\upsilon )+A_{2}(x,\upsilon ).
\end{eqnarray*}%
Obviously,%
\begin{equation*}
|A_{1}(x\,,\upsilon )|\leq K[\int_{|y|>1}|\zeta (x+y-\upsilon )|\frac{dy}{%
|y|^{d+\alpha }}+|\zeta (x-\upsilon )|\int_{|y|>1}\frac{dy}{|y|^{d+\alpha }}]
\end{equation*}%
and%
\begin{equation*}
|A_{2}(x,\upsilon )|\leq K\int_{|y|\leq 1}\int_{0}^{1}|\nabla \zeta
(x+sy-\upsilon )|\frac{dy}{|y|^{d+\alpha -1}}.
\end{equation*}%
Therefore,%
\begin{equation*}
|\mathcal{A}_{x}\mathcal{\zeta }(x-\cdot )|_{p}^{p}\leq CK^{p}.
\end{equation*}%
Denoting%
\begin{equation*}
D(x,\upsilon )=\int \left( u(x+y)-u(x)\right) \left( \zeta (x+y-\upsilon
)-\zeta (x-\upsilon )\right) m(x,y)\frac{dy}{|y|^{d+\alpha }},
\end{equation*}%
we have%
\begin{eqnarray*}
|D(x,\upsilon )| &\leq &K\int_{0}^{1}\int_{|y|\leq 1}|u(x+y)-u(x)|~|\nabla
\zeta (x+sy-\upsilon )|\frac{dy}{|y|^{d+\alpha -1}} \\
&&+K\int_{|y|>1}(|u(x+y)|+|u(x)|\left( |\zeta (x+y-\upsilon )|+|\zeta
(x-\upsilon )|\right) \frac{dy}{|y|^{d+\alpha }}
\end{eqnarray*}%
and%
\begin{equation*}
\int \int |D(x,\upsilon )|^{p}dxd\upsilon \leq CK^{p}|u|_{\alpha ^{\prime
},p}^{p}
\end{equation*}%
for some $\alpha ^{\prime }<\alpha $. Therefore, by Lemma \ref{rem6},%
\begin{equation*}
|\mathcal{A}u|_{p}^{p}\leq CK^{p}[\int |\left( u(\cdot )\zeta (\cdot
-\upsilon )\right) |_{\alpha ,p}^{p}d\upsilon +|u|_{p}^{p}+|u|_{\alpha
^{\prime },p}^{p}).
\end{equation*}%
Since as in (\ref{fo6}) 
\begin{eqnarray*}
&&\partial _{x}^{\alpha }\left( u(x)\zeta (x-\upsilon )\right) \\
&=&\partial ^{\alpha }u(x)\zeta (x-\upsilon )+u(x)\partial ^{\alpha }\zeta
(x-\upsilon ) \\
&&+\int \left( u(x+y)-u(x)\right) \left( \zeta (x+y-\upsilon )-\zeta
(x-\upsilon )\right) \frac{dy}{|y|^{d+\alpha }},
\end{eqnarray*}%
we derive in a similar way,%
\begin{eqnarray*}
\int |\left( u(\cdot )\zeta (\cdot -\upsilon )\right) |_{\alpha
,p}^{p}d\upsilon &=&\int |\left( u(\cdot )\zeta (\cdot -\upsilon )\right)
|_{p}^{p}d\upsilon +\int |\partial ^{\alpha }(u(\cdot )\zeta (\cdot
-\upsilon ))|_{p}^{p}d\upsilon \\
&\leq &C|u|_{\alpha ,p}.
\end{eqnarray*}%
The statement follows.
\end{proof}

\subsection{Solution for $m$ independent of spacial variable}

In this section, we consider the following partial case of equation (\ref%
{eq1}):%
\begin{eqnarray}
\partial _{t}u(t,x) &=&A^{m}u(t,x)-\lambda u(t,x)+f(t,x),  \label{pr4} \\
u(0,x) &=&0,  \notag
\end{eqnarray}%
where $m(t,x,y)=m(t,y)$ does not depend on the spacial variable.

We denote by $\mathfrak{D}_{p}(E)$, $p\geq 1$, the space of all measurable
functions $f$ on $E$ such that $f\in \cap _{\kappa >0}H_{p}^{\kappa }(E)$
and for every multiindex \hbox{$\gamma \in
\mathbf{N}_{0}^{d}$}%
\begin{equation*}
\sup_{(t,x)\in H}|D_{x}^{\gamma }f(t,x)|<\infty .
\end{equation*}

The set $\mathfrak{D}_{p}(E)$ is a dense subset of $H_{p}^{\kappa }(E)$ (see 
\cite{mikprag}).

\begin{lemma}
\label{le3}(see Theorem 14 in \cite{mikprag}) Let$\ p\geq 2,f\in \mathfrak{D}%
_{p}(E)$ and Assumption~\emph{A} be satisfied.

Then there is a unique strong solution $u\in \mathfrak{D}_{p}(E)$ of \emph{(%
\ref{pr4})}. Moreover, $u(t,x)$ is continuous in $t$, smooth in $x$ and the
following assertions hold:

\emph{(i)} for every multiindex $\gamma \in \mathbf{N}_{0}^{d}$ 
\begin{equation}
|D^{\gamma }u|_{p}\leq C\rho _{\lambda }|D^{\gamma }f|_{p}.  \notag
\end{equation}%
where $\rho _{\lambda }=T\wedge \frac{1}{\lambda }$ and the constant $%
C=C(\alpha ,p,d,K,\eta )$;

\emph{(ii)} the following estimate holds:%
\begin{equation*}
|u|_{\alpha ,p}\leq C|f|_{p},
\end{equation*}%
where the constant $C=C(\alpha ,p,d,T,K,\eta )$.
\end{lemma}

Passing to the limit we arrive at

\begin{proposition}
\label{prop1}Let $p\geq 2,f\in L_{p}(E)$ and Assumption~\emph{A} be
satisfied.

Then there is a unique strong solution $u\in H_{p}^{\alpha }(E)$ of \emph{(%
\ref{pr4})}. Moreover, there are constants $C_{0}=C_{0}(\alpha ,p,d,T,K,\eta
)$ and $C_{00}=C_{00}(\alpha ,p,d,K,\eta )$ such that 
\begin{equation}
|u|_{\alpha ,p}\leq C_{0}|f|_{p}  \notag
\end{equation}%
and%
\begin{equation}
|u|_{p}\leq C_{00}\rho _{\lambda }|f|_{p}  \notag
\end{equation}%
where $\rho _{\lambda }=T\wedge \frac{1}{\lambda }$.
\end{proposition}

\begin{proof}
\emph{Existence}. There is a sequence of input functions $f_{n},$ $%
n=1,2,\ldots ,$ such that $f_{n}\in \mathfrak{D}_{p}(E)$, and%
\begin{equation}
|f-f_{n}|_{p}\rightarrow 0  \label{eq24}
\end{equation}%
as $n\rightarrow \infty $. By Lemma \ref{le3}, for every $n$ there is a
strong solution $u_{n}\in \mathfrak{D}_{p}(E)$ of (\ref{pr4}) with the input
function $f_{n}$. Since (\ref{pr4}) is a linear equation, using the estimate
(ii) of Lemma~\ref{le3} we derive that $(u_{n})$ is a Cauchy sequence in $%
H_{p}^{\alpha }(E)$. Hence, there is a function $u\in H_{p}^{\alpha }(E)$
such that $|u_{n}-u|_{\alpha ,p}\rightarrow 0$ as $n\rightarrow \infty $.

Passing to the limit in the inequalities of Lemma \ref{le3} with $u,f$
replaced by $u_{n},f_{n}$ ($\gamma =0)$, we get the corresponding estimates
for $u.$

Denoting $\left\langle f,g\right\rangle =\int fgdx$ and passing to the limit
in the equality (see definition (\ref{defs})) 
\begin{equation*}
\left\langle u_{n}(t,\cdot ),\varphi \right\rangle =\int_{0}^{t}\Bigl[%
\big\langle(A-\lambda )u_{n}(s,\cdot )+f(s,\cdot ),\varphi \big\rangle\Bigr]%
ds,\varphi \in \mathcal{S}(\mathbf{R}^{d}),
\end{equation*}%
as $n\rightarrow \infty $, we get that the function $u$ is a weak solution
of (\ref{pr4}).

Since for each $\upsilon \in H_{p}^{\alpha }(E)$ 
\begin{equation*}
|A\upsilon |_{\alpha ,p}\leq C|\upsilon |_{p},
\end{equation*}%
the solution is strong.

\medskip \emph{Uniqueness}. Let $u\in H_{p}^{\alpha }(E)$ be a solution of (%
\ref{pr4}) with zero input function $f$. Hence, for every $\varphi \in 
\mathcal{S}(\mathbf{R}^{d})$ and $t\in \lbrack 0,T]$ 
\begin{equation}
\big\langle u(t,\cdot ),\varphi \big\rangle=\int_{0}^{t}\big\langle %
u(s,\cdot ),A^{(\alpha )\ast }\varphi -\lambda \varphi \big\rangle ds
\label{eq25}
\end{equation}

Let $\zeta _{\varepsilon }=\zeta _{\varepsilon }(x)$, $x\in \mathbf{R}^{d}$, 
$\varepsilon \in (0,1)$, be a standard mollifier. Inserting $\varphi (\cdot
)=\zeta _{\varepsilon }(x-{\cdot })$ into (\ref{eq25}), we get that the
function 
\begin{equation*}
\upsilon _{\varepsilon }(t,x)=u(t,\cdot )\ast \zeta _{\varepsilon }(x)
\end{equation*}%
belongs to $\mathfrak{D}_{p}(E)$ and 
\begin{equation*}
\upsilon _{\varepsilon }(t,x)=\int_{0}^{t}\big(A-\lambda \big)\upsilon
_{\varepsilon }(s,x)ds.
\end{equation*}%
By Lemma \ref{le3}, $\upsilon _{\varepsilon }=0$ in $E$ for all $\varepsilon
\in (0,1)$. Hence, $u(t,\cdot )=0$ and the statement holds.
\end{proof}

\section{Proofs of main Theorems}

We follow the proof of Theorem 1.6.4 in \cite{kry}. In order to use the
method of continuity, we derive the a priori estimates first.

\begin{lemma}
\label{lem3}Assume \textbf{A} holds, $\beta \in (0,1),p>d/\beta ,p\geq 2.$
There are $\varepsilon =\varepsilon (d,\alpha ,\beta ,K,w,T,\eta
),C=C(d,\alpha ,\beta ,p,K,w,T,\eta )$ and $\lambda _{0}=\lambda
_{0}(d,\alpha ,\beta ,p,K,w,T,\eta )\geq 1$ such that for any $u\in 
\mathfrak{D}_{p}(E)$ satisfying (\ref{eq1}) with $B=0$ and with support in a
ball of radius $\varepsilon $ ($u(t,x)=0$ for all $t$ if $x$ does not belong
to a ball of radius $\varepsilon $),%
\begin{eqnarray*}
|u|_{\alpha ,p} &\leq &C|f|_{p}, \\
|u|_{p} &\leq &\frac{C}{\lambda }|f|_{p}\text{ if }\lambda \geq \lambda _{0}.
\end{eqnarray*}
\end{lemma}

\begin{proof}
Let the support of $u$ be a subset of the ball centered at $x_{0}$ with
radius $\varepsilon >0$. Then%
\begin{eqnarray*}
\partial _{t}u &=&A_{t,x_{0}}u(t,x)+A_{t,x}u(t,x)-A_{t,x_{0}}u(t,x)-\lambda
u+f, \\
u(0) &=&0.
\end{eqnarray*}%
Let $\varphi \in C_{0}^{\infty }(\mathbf{R}^{d}),0\leq \varphi \leq
1,\varphi (x)=1$ if $|x|\leq 1,$ and $\varphi (x)=0$ if $|x|>2$. Denote 
\begin{eqnarray*}
\tilde{A} &=&\varphi (\frac{x-x_{0}}{2\varepsilon }%
)[A_{t,x}u(t,x)-A_{t,x_{0}}u(t,x)], \\
m_{0}(t,x,y) &=&m(t,x,y)-m(t,x_{0},y),(t,x)\in E,y\in \mathbf{R}^{d}.
\end{eqnarray*}%
By Corollary \ref{cor4},%
\begin{equation}
|\tilde{A}|_{p}\leq C|u|_{\alpha ,p}K_{\varepsilon },  \label{fo7}
\end{equation}%
where $C=C(\alpha ,\beta ,p,d)$ and $K_{\varepsilon }$ is the constant
bounding%
\begin{equation*}
M(t,z,y)=|m_{0}(t,z,y)\varphi (\frac{z-x_{0}}{2\varepsilon })|+|\partial
^{\beta }\left( m_{0}(t,z,y)\varphi (\frac{z-x_{0}}{2\varepsilon })\right) |.
\end{equation*}%
Obviously, $|m_{0}(t,z,y)\varphi (\frac{z-x_{0}}{2\varepsilon })|\leq
w(2\varepsilon ),z,y\in R^{d},t\in \lbrack 0,T],$ and%
\begin{eqnarray*}
&&|\partial ^{\beta }\left( m_{0}(t,z,y)\varphi (\frac{z-x_{0}}{2\varepsilon 
})\right) | \\
&\leq &\int_{|\upsilon |>\varepsilon }\left\vert m_{0}(t,z+\upsilon
,y)\varphi (\frac{z+\upsilon -x_{0}}{2\varepsilon })-m_{0}(t,z,y)\varphi (%
\frac{z-x_{0}}{2\varepsilon })\right\vert \frac{d\upsilon }{|\upsilon
|^{1+\beta }} \\
&&+|\varphi (\frac{z-x_{0}}{2\varepsilon })|\int_{|\upsilon |\leq
\varepsilon }[m(t,z+\upsilon ,y)-m(t,z,y)]\frac{d\upsilon }{|\upsilon
|^{d+\beta }} \\
&&+|m_{0}(t,z,y)|\int_{|\upsilon |\leq \varepsilon }\left\vert \varphi (%
\frac{z+\upsilon -x_{0}}{2\varepsilon })-\varphi (\frac{z-x_{0}}{%
2\varepsilon })\right\vert \frac{d\upsilon }{|\upsilon |^{d+\beta }} \\
&&+|\int_{|\upsilon |\leq \varepsilon }\left\vert \varphi (\frac{z+\upsilon
-x_{0}}{2\varepsilon })-\varphi (\frac{z-x_{0}}{2\varepsilon })\right\vert
\left\vert m(t,z+\upsilon ,y)-m(t,z,y)\right\vert \frac{d\upsilon }{%
|\upsilon |^{d+\beta }} \\
&\leq &C[w(2\varepsilon )\varepsilon ^{-\beta }+\int_{|\upsilon |\leq
\varepsilon }w(|\upsilon |)\frac{d\upsilon }{|\upsilon |^{d+\beta }}].
\end{eqnarray*}%
Therefore%
\begin{equation*}
K_{\varepsilon }\leq C[w(2\varepsilon )\varepsilon ^{-\beta
}+\int_{|\upsilon |\leq \varepsilon }w(|\upsilon |)\frac{d\upsilon }{%
|\upsilon |^{d+\beta }}]
\end{equation*}%
and $K_{\varepsilon }\rightarrow 0$ as $\varepsilon \rightarrow 0$.
Obviously,%
\begin{eqnarray*}
|Au-\tilde{A}|_{p} &\leq &C\int_{|y|>\varepsilon }|\nabla _{y}^{\alpha
}u(\cdot )|_{p}\frac{dy}{y^{d+\alpha }} \\
&\leq &C\varepsilon ^{-\alpha }[|u|_{p}+1_{\alpha >1}|\nabla u|_{p}].
\end{eqnarray*}%
So, by Proposition \ref{prop1} and (\ref{fo7}), there are constants $%
C_{1}=C_{1}(\alpha ,p,d,T,K,\eta )$ and $C_{11}=C_{11}(\alpha ,p,d,K)$ such
that 
\begin{equation*}
|u|_{\alpha ,p}\leq C_{1}\left[ |f|_{p}+K_{\varepsilon }|u|_{\alpha
,p}+\varepsilon ^{-\alpha }(|u|_{p}+1_{\alpha >1}|\nabla u|_{p})\right]
\end{equation*}%
with $K_{\varepsilon }\rightarrow 0$ as $\varepsilon \rightarrow 0$ and%
\begin{equation*}
|u|_{p}\leq C_{11}\rho _{\lambda }[|f|_{p}+K_{\varepsilon }|u|_{\alpha
,p}+\varepsilon ^{-\alpha }(|u|_{p}+1_{\alpha >1}|\nabla u|_{p})],
\end{equation*}%
where $\rho _{\lambda }=\frac{1}{\lambda }\wedge T$. We choose $\varepsilon $
so that $C_{1}K_{\varepsilon }\leq 1/2,K_{\varepsilon }\leq 1$. In this case,%
\begin{eqnarray*}
|u|_{\alpha ,p} &\leq &2C_{1}\left[ |f|_{p}+\varepsilon ^{-\alpha
}(|u|_{p}+1_{\alpha >1}|\nabla u|_{p})\right] , \\
|u|_{p} &\leq &C_{11}(1+2C_{1})\rho _{\lambda }[|f|_{p}+\varepsilon
^{-\alpha }(|u|_{p}+1_{\alpha >1}|\nabla u|_{p})].
\end{eqnarray*}

By interpolation inequality, for $\alpha >1$ and each $\kappa \in (0,1)$
there is a constant $C=C(\alpha ,p,d)$ such that%
\begin{equation*}
|\nabla u|_{p}\leq \kappa |u|_{\alpha ,p}+C\kappa ^{-\frac{1}{\alpha -1}%
}|u|_{p}\text{.}
\end{equation*}%
Therefore choosing $\kappa $ so that $2C_{1}\varepsilon ^{-\alpha }\kappa
\leq \frac{1}{2}$ (if $\alpha >1)$, one can see that there is $\tilde{C}_{1}=%
\tilde{C}_{1}(\alpha ,\beta ,p,d,T,K,w,\eta )$ such that 
\begin{eqnarray*}
|u|_{\alpha ,p} &\leq &\tilde{C}_{1}\left[ |f|_{p}+|u|_{p}\right] , \\
|u|_{p} &\leq &\tilde{C}_{1}\rho _{\lambda }\left[ |f|_{p}+|u|_{p}\right] .
\end{eqnarray*}%
The statement follows by choosing $\lambda $ so that $\tilde{C}_{1}\lambda
^{-1}\leq \frac{1}{2}$ \NEG{(}$\lambda _{0}=(2\tilde{C}_{1})^{-1}).$
\end{proof}

Now we extend the estimates.

\begin{lemma}
\label{lem4}Assume \textbf{A} holds and $p>\frac{d}{\beta },p\geq 2$. There
is a constant $C=C(d,\alpha ,\beta ,p,K,w,T,\eta )$ and a number $\lambda
_{1}=\lambda _{1}(\alpha ,\beta ,d,K,w,\eta ,T)>1$ such that for any $u\in 
\mathfrak{D}_{p}(E)$ satisfying (\ref{eq1}) with $B=0$ and $\lambda \geq
\lambda _{1}$, 
\begin{eqnarray*}
|u|_{\alpha ,p} &\leq &C|f|_{p}, \\
|u|_{p} &\leq &\frac{C}{\lambda }|f|_{p}\text{ if }\lambda \geq \lambda _{1}.
\end{eqnarray*}
\end{lemma}

\begin{proof}
As in \cite{kry}, Theorem 1.6.4, take $\zeta \in C_{0}^{\infty }(\mathbf{R}%
^{d})$ such that $\int |\zeta |^{p}dx=1$ and whose support is in a ball of
radius $\varepsilon $ from Lemma \ref{lem3} centered at $0.$ Then%
\begin{equation}
|\partial ^{\alpha }u(t,x)|^{p}=\int |\partial ^{\alpha }u(t,x)\zeta
(x-\upsilon )|^{p}d\upsilon  \label{fo9}
\end{equation}%
and%
\begin{eqnarray}
&&\partial ^{\alpha }u(t,x)\zeta (x-\upsilon )  \label{fo8} \\
&=&\partial ^{\alpha }\left( u(t,x)\zeta (x-\upsilon )\right)
-u(t,x)\partial _{x}^{\alpha }\zeta (x-\upsilon )  \notag \\
&&+\int \left[ u(t,x+y)-u(t,x)\right] \left[ \zeta (x+y-\upsilon )-\zeta
(x-\upsilon )\right] \frac{dy}{|y|^{d+\alpha }}.  \notag
\end{eqnarray}%
Since 
\begin{eqnarray*}
&&\partial _{t}\left( u(t,x)\zeta (x-\upsilon )\right) \\
&=&\zeta (x-\upsilon )Au(t,x)-\lambda \zeta (x-\upsilon )u(t,x)+\zeta
(x-\upsilon )f(t,x) \\
&=&A\left( \zeta (x-\upsilon )u(t,x)\right) -\lambda \zeta (x-\upsilon
)u(t,x)+\zeta (x-\upsilon )f(t,x) \\
&&-u(t,x)A\zeta (x-\upsilon ) \\
&&-\int \left[ u(t,x+y)-u(t,x)\right] \left[ \zeta (x+y-\upsilon )-\zeta
(x-\upsilon )\right] m(t,x,y)\frac{dy}{|y|^{d+\alpha }},
\end{eqnarray*}%
it follows by Lemma \ref{lem3} that there is $C=C(d,\alpha ,\beta
,p,K,w,T,\eta )$ and $\lambda _{0}=\lambda _{0}(d,\alpha ,\beta
,p,K,w,T,\eta )$ such that%
\begin{eqnarray*}
\int |u\zeta (\cdot -\upsilon )|_{\alpha ,p}^{p}d\upsilon &\leq
&C[|f|_{p}^{p}+|u|_{p}^{p}+|u|_{\alpha ^{\prime },p}^{p}], \\
\int |u\zeta (\cdot -\upsilon )|_{p}^{p}d\upsilon &\leq &\frac{C}{\lambda
^{p}}[|f|_{p}^{p}+|u|_{p}^{p}+|u|_{\alpha ^{\prime },p}^{p}]\text{ if }%
\lambda \geq \lambda _{0},
\end{eqnarray*}%
for some $\alpha ^{\prime }<\alpha $. According to (\ref{fo8}) and (\ref{fo9}%
),%
\begin{eqnarray}
|\partial ^{\alpha }u|_{p}^{p} &\leq &C[|f|_{p}^{p}+|u|_{p}^{p}+|u|_{\alpha
^{\prime },p}^{p}],  \label{for9} \\
|u|_{p}^{p} &\leq &\frac{C}{\lambda ^{p}}[|f|_{p}^{p}+|u|_{p}^{p}+|u|_{%
\alpha ^{\prime },p}^{p}]\text{ if }\lambda \geq \lambda _{0}.  \notag
\end{eqnarray}%
By interpolation inequality, for each $\kappa >0$ there is a constant $%
K_{1}=K_{1}(\kappa ,\alpha ^{\prime },\alpha ,p,d)$ such that%
\begin{equation*}
|u|_{\alpha ^{\prime },p}\leq \kappa |u|_{\alpha ,p}+K_{1}|u|_{p}.
\end{equation*}%
Therefore, choosing $\kappa $ so that $C\kappa \leq 1/2$, we get by (\ref%
{for9}) that there is a constant $C_{1}=C_{1}(d,\alpha ,\beta ,p,K,w,T,\eta
) $ such that 
\begin{eqnarray}
|u|_{\alpha ,p}^{p} &\leq &C_{1}[|f|_{p}^{p}+|u|_{p}],  \label{for10} \\
|u|_{p}^{p} &\leq &\frac{C_{1}}{\lambda ^{p}}[|f|_{p}^{p}+|u|_{p}^{p}]. 
\notag
\end{eqnarray}%
We finish the proof by choosing $\lambda $ so that $\frac{C_{1}}{\lambda ^{p}%
}\leq \frac{1}{2}$ or $\lambda \geq (2C_{1})^{1/p}=\lambda _{1}$. Thus by (%
\ref{for10}), 
\begin{equation*}
|u|_{p}^{p}\leq \frac{2C_{1}}{\lambda ^{p}}|f|_{p}^{p},|u|_{\alpha
,p}^{p}\leq C_{1}(1+\frac{2C_{1}}{\lambda _{1}^{p}})|f|_{p}^{p}.
\end{equation*}

The statement follows.
\end{proof}

\begin{corollary}
\label{cor5}Assume \textbf{A} holds, $p>\frac{d}{\beta },p\geq 2$ and $u\in 
\mathfrak{D}_{p}(E)$ satisfies (\ref{eq1}$)$ with $B=0$. Then there is $%
C=C(d,\alpha ,\beta ,p,K,w,T,\eta )$ such that 
\begin{equation*}
|u|_{\alpha ,p}\leq C|f|_{p}.
\end{equation*}
\end{corollary}

\begin{proof}
For $\lambda \geq \lambda _{1}$ ($\lambda _{1}$ is from Lemma \ref{lem4}$)$,
the estimate is proved in Lemma \ref{lem4}$.$ If $u\in H_{p}^{\alpha }(E)$
solves (\ref{eq1}) with $\lambda \leq \lambda _{1}$, then $\tilde{u}%
(t,x)=e^{(\lambda _{1}-\lambda )t}u(t,x)$ solves the same equation with $%
\lambda =\lambda _{1}$ and $f$ replaced by $e^{(\lambda _{1}-\lambda )t}f.$
Hence%
\begin{equation*}
|u|_{\alpha ,p}\leq |\tilde{u}|_{\alpha ,p}\leq Ce^{(\lambda _{1}-\lambda
)T}|f|_{p}
\end{equation*}%
with $C=C(d,\alpha ,\beta ,p,K,w,T,\eta )$ from Lemma \ref{lem4}. So, the
estimate holds for all $\lambda \geq 0$.
\end{proof}

\subsection{Proof of Theorem \protect\ref{main}}

We use the a priori estimate and the continuation by parameter argument. Let 
\begin{equation*}
M_{\tau }u=\tau Lu+\left( 1-\tau \right) \partial ^{\alpha }u,\tau \in \left[
0,1\right] .
\end{equation*}%
We introduce the space $\mathcal{H}_{p}^{\alpha }\left( E\right) $ of
functions $u\in H_{p}^{\alpha }(E)$ such that for each, $u\left( t,x\right)
=\int_{0}^{t}F\left( s,x\right) \,ds,$where $F\in H_{p}^{\alpha }\left(
E\right) .$ It is a Banach space with respect to the norm 
\begin{equation*}
||u||_{\alpha ,p}=\left\vert u\right\vert _{\alpha ,p}+\left\vert
F\right\vert _{p}.
\end{equation*}%
Consider the mappings $T_{\tau }:\mathcal{H}_{p}^{\alpha }\left( E\right)
\rightarrow L_{p}(E)$ defined by%
\begin{equation*}
u\left( t,x\right) =\int_{0}^{t}F\left( s,x\right) \,ds\longmapsto F-M_{\tau
}u.
\end{equation*}%
Obviously, for some constant $C$ not depending on $\tau ,$ 
\begin{equation*}
\left\vert T_{\tau }u\right\vert _{p}\leq C|\left\vert u\right\vert
|_{\alpha ,p}.
\end{equation*}%
On the other hand, there is a constant $C$ not depending on $\tau $ such
that for all $u\in \mathcal{H}_{p}^{\alpha }\left( E\right) $%
\begin{equation}
|\left\vert u\right\vert |_{\alpha ,p}\leq C\left\vert T_{\tau }u\right\vert
_{p}.  \label{cp1}
\end{equation}%
Indeed, 
\begin{equation*}
u\left( t,x\right) =\int_{0}^{t}F\left( s,x\right) \,ds=\int_{0}^{t}\left(
M_{\tau }u+(F-M_{\tau }u)\right) (s,x)\,ds,
\end{equation*}%
and, according to Corollary \ref{cor5}, there is a constant $C$ not
depending on $\tau $ such that 
\begin{equation}
\left\vert u\right\vert _{\alpha ,p}\leq C\left\vert T_{\tau }u\right\vert
_{p}=C\left\vert F-M_{\tau }u\right\vert _{p}.  \label{cp2}
\end{equation}%
Thus,%
\begin{eqnarray*}
|\left\vert u|\right\vert _{\alpha ,p} &=&\left\vert u\right\vert _{\alpha
,p}+\left\vert F\right\vert _{p}\leq \left\vert u\right\vert _{\alpha
,p}+\left\vert F-M_{\tau }u\right\vert _{p}+\left\vert M_{\tau }u\right\vert
_{p} \\
&\leq &C\left( \left\vert u\right\vert _{\alpha ,p}+\left\vert F-M_{\tau
}u\right\vert _{p}\right) \leq C\left\vert F-M_{\tau }u\right\vert
_{p}=C\left\vert T_{\tau }u\right\vert _{p},
\end{eqnarray*}%
and (\ref{cp1}) follows. Since $T_{0}$ is an onto map, by Theorem 5.2 in 
\cite{GiT83} all the $T_{\tau }$ are onto maps and Theorem \ref{main}
follows.

\subsection{Proof of Theorem \protect\ref{main2}}

Assume \textbf{A} holds and $p>\frac{d}{\beta }\vee \frac{d}{\alpha },p\geq
2 $ and $u\in \mathfrak{D}_{p}(E)$ satisfies (\ref{eq1}) with $%
B=B^{\varepsilon _{0}}$. By Theorem \ref{main}, 
\begin{eqnarray*}
|\partial _{t}u|_{p}+|u|_{\alpha ,p} &\leqslant &N[|f|_{p}+|B^{\varepsilon
_{0}}u|_{p}], \\
|u|_{p} &\leq &\frac{N}{\lambda }[|f|_{p}+|B^{\varepsilon _{0}}u|_{p}]\text{
if }\lambda \geq \lambda _{1},
\end{eqnarray*}%
where $\lambda _{1}=\lambda _{1}(T,\alpha ,\beta ,d,K,w,\eta )\geq 1$.
According to Lemma \ref{lem0},%
\begin{eqnarray}
|\partial _{t}u|_{p}+|u|_{\alpha ,p} &\leqslant &2N|f|_{p},  \label{f2} \\
|u|_{p} &\leq &\frac{2N}{\lambda }|f|_{p}\text{ if }\lambda \geq \lambda
_{1}.  \notag
\end{eqnarray}

If $u\in \mathfrak{D}_{p}(E)$ satisfies (\ref{eq1}) with $B=B^{\varepsilon
_{0}},\lambda \leq \lambda _{1}$, then $\tilde{u}(t,x)=e^{(\lambda
_{1}-\lambda )t}u(t,x)$ satisfies the same equation with $\lambda _{1}$ with 
$f$ replaced by $e^{(\lambda _{1}-\lambda )t}f$. By (\ref{f2}),%
\begin{equation}
|u|_{\alpha ,p}\leq |\tilde{u}|_{\alpha ,p}\leq 2Ne^{(\lambda _{1}-\lambda
)T}|f|_{p}.  \label{f3}
\end{equation}

The statement follows by the a priori estimates (\ref{f2})-(\ref{f3}) and
the continuation by parameter argument, repeating the proof of Theorem \ref%
{main} for the operators%
\begin{equation*}
M_{\tau }=A+\tau B^{\varepsilon _{0}},0\leq \tau \leq 1.
\end{equation*}

\subsection{Proof of Theorem \protect\ref{main3}}

Again we derive the a priori estimates first and use the continuation by
parameter argument. There is $\varepsilon _{0}\in (0,1)$ such that 
\begin{equation*}
\int_{|y|\leq \varepsilon _{0}}|y|^{\alpha }\pi (t,x,dy)\leq \delta
_{0},(t,x)\in E,
\end{equation*}%
where $\delta _{0}$ is a number in Theorem \ref{main2}. Let $u\in \mathfrak{D%
}_{p}(E)$ satisfy (\ref{eq1}). Let 
\begin{equation*}
\tilde{L}v(t,x)=Av(t,x)+B^{\varepsilon _{0}}v(t,x),
\end{equation*}%
$(t,x)\in E,$ where $B^{\varepsilon _{0}}v$ is defined in (\ref{f0}).
Applying Theorem \ref{main2} to $\tilde{L}$, we have%
\begin{eqnarray*}
|\partial _{t}u|_{p}+|u|_{\alpha ,p} &\leqslant &N_{2}[|f|_{p}|+|(L-\tilde{L}%
)u|_{p}], \\
|u|_{p} &\leq &\frac{N_{2}}{\lambda }[|f|_{p}+(L-\tilde{L})u|_{p}]\text{ if }%
\lambda \geq \lambda _{2}.
\end{eqnarray*}%
There is $\alpha ^{\prime }<\alpha $ such that $p>d/\alpha ^{\prime }$ and
by Sobolev embedding theorem there is a constant $C$ such that%
\begin{equation*}
|(L-\tilde{L})u|_{p}\leq C|u|_{\alpha ^{\prime },p}\left( \int_{0}^{T}\int
\pi (t,x,\left\{ |y|>\varepsilon _{0}\right\} )^{p}dxdt\right) ^{1/p}
\end{equation*}%
By interpolation inequality, for each $\kappa >0$ there is a constant $%
\tilde{N}=\tilde{N}(\kappa ,K_{2},\alpha ,\alpha ^{\prime },p,d)$ such that%
\begin{equation*}
|(L-\tilde{L})u|_{p}\leq \kappa |u|_{\alpha ,p}+\tilde{N}|u|_{p},
\end{equation*}%
where $K_{2}$ is a constant bounding $\int_{0}^{T}\int \pi (t,x,\left\{
|y|>\varepsilon _{0}\right\} )^{p}dxdt$. Choose $\kappa $ so that $2N\kappa
\leq 1/2$. Then%
\begin{eqnarray*}
|\partial _{t}u|_{p}+|u|_{\alpha ,p} &\leqslant &4N[|f|_{p}|+\tilde{N}%
|u|_{p}], \\
|u|_{p} &\leq &\text{ }\frac{4N}{\lambda }[|f|_{p}+\tilde{N}|u|_{p}]\text{
if }\lambda \geq \lambda _{1}.
\end{eqnarray*}%
Choosing $\lambda \geq \lambda _{2}=8N\tilde{N},$ we derive%
\begin{eqnarray*}
|u|_{p} &\leq &\text{ }\frac{8N}{\lambda }|f|_{p}, \\
|\partial _{t}u|_{p}+|u|_{\alpha ,p} &\leqslant &8N|f|_{p}|.
\end{eqnarray*}

Multiplying\bigskip\ $u$ by $e^{(\lambda -\lambda _{2})t}$, we obtain the a
priori estimate for all $\lambda \geq 0$ as in the proof of Corollary \ref%
{cor5} above.

The statement follows by the a priori estimates and the continuation by
parameter argument, repeating the proof of Theorem \ref{main} for the
operators%
\begin{equation*}
M_{\tau }v=\tilde{L}v+\tau (L-\tilde{L})v,0\leq \tau \leq 1.
\end{equation*}

\section{Embedding of the solution space}

Following the main steps of Section 7 in \cite{kry1}, we will show that for
a sufficiently large $p$, the H\"{o}lder norm of the solution is finite.
Since the solution of (\ref{eq1}) $u\in \mathcal{H}_{p}^{\alpha }(E),$ we
will derive an embedding theorem for $\mathcal{H}_{p}^{\alpha }(E)$.

\begin{remark}
\label{rem7}If $u\in \mathcal{H}_{p}^{\alpha }(E)$, then $u\in H_{p}^{\alpha
}(E)$ and\ 
\begin{equation*}
u(t)=\int_{0}^{t}F(s)ds,0\leq t\leq T,
\end{equation*}%
with $F\in L_{p}(E)$. It is the $\mathcal{H}_{p}^{\alpha }$-solution to the
equation 
\begin{eqnarray}
\partial _{t}u &=&\partial ^{\alpha }u+f,  \label{fn0} \\
u(0) &=&0,  \notag
\end{eqnarray}%
where $f=F-\partial ^{\alpha }u\in L_{p}(E)$ with $|f|_{p}\leq
|F|_{p}+|\partial ^{\alpha }u|_{p}\leq ||u||_{\alpha ,p}$. In addition
(e.g., see \cite{mikprag}), 
\begin{equation}
u(t,x)=\int_{0}^{t}G_{t-s}(x-y)f(s,y)dyds,0\leq t\leq T,x\in \mathbf{R}^{d},
\label{fn}
\end{equation}%
where 
\begin{equation}
G_{t}=\mathcal{F}^{-1}\left[ e^{-t|\xi |^{\alpha }}\right] ,t>0,  \label{fn1}
\end{equation}%
(here $\mathcal{F}^{-1}$ is the inverse Fourier transform). The function $%
G_{t}$ is the probability density function of a spherically symmetric $%
\alpha $-stable process whose generator is the fractional Laplacian $%
\partial ^{\alpha }:$%
\begin{equation}
\int G_{t}dx=1,t>0.  \label{fnew}
\end{equation}
\end{remark}

\begin{remark}
\label{rem8}Note that for any multiindex $\gamma \in \mathbf{N}_{0}^{d}$
there is a constant $C=C(\alpha ,\gamma ,d)$ such that%
\begin{equation}
|D_{\xi }^{\gamma }e^{-|\xi |^{\alpha }}|\leq Ce^{-|\xi |^{\alpha
}}\sum_{1\leq k\leq |\gamma |}|\xi |^{k\alpha -|\gamma |}.  \label{fn2}
\end{equation}
\end{remark}

\begin{lemma}
\label{lemn1}Let $K(x)=G_{1}(x),x\in \mathbf{R}^{d}$. Then

(i) $K$ is smooth and for all multiindices $\gamma \in \mathbf{N}%
_{0}^{d},\kappa \in (0,2)$,%
\begin{equation*}
\int |\partial ^{\kappa }D^{\gamma }K(x)|dx<\infty .
\end{equation*}

(ii) for $t>0,x\in \mathbf{R}^{d},$%
\begin{equation*}
G_{t}(x)=t^{-d/\alpha }K(x/t^{1/\alpha })
\end{equation*}%
and for any multiindex $\gamma \in \mathbf{N}_{0}^{d}$ and $\kappa \in
(0,2), $ there is a constant $C$ such that%
\begin{equation*}
|\partial ^{\kappa }D^{\gamma }G_{t}\ast v|_{p}\leq Ct^{-(|\gamma |+\kappa
)/\alpha }|v|_{p},t>0,v\in L_{p}(\mathbf{R}^{d}).
\end{equation*}

(iii) Let $\kappa \in (0,1)$. There is a constant $C$ such that for $v\in 
\mathcal{S}(\mathbf{R}^{d}),t>0,$%
\begin{equation*}
|G_{t}\ast v-v|_{p}\leq Ct^{\kappa }|\partial ^{\alpha \kappa }v|_{p}.
\end{equation*}
\end{lemma}

\begin{proof}
(i) For any multiinidex $\gamma \in \mathbf{N}_{0}^{d},$ 
\begin{equation*}
\sup_{x}|D^{\gamma }K(x)|\leq \int |\left( i\xi \right) ^{\gamma }e^{-|\xi
|^{\alpha }}|d\xi <\infty .
\end{equation*}%
Let $\varphi \in C_{0}^{\infty }(\mathbf{R}^{d}),0\leq \varphi \leq
1,\varphi (x)=1$ if $|x|\leq 1,\varphi (x)=0$ if $|x|\geq 2$. Then $%
K(x)=K_{1}(x)+K_{2}(x)$ with%
\begin{equation*}
K_{1}=\mathcal{F}^{-1}\left( e^{-|\xi |^{\alpha }}\varphi (\xi )\right)
,K_{2}=\mathcal{F}^{-1}\left( [1-\varphi (\xi )]e^{-|\xi |^{\alpha }}\right)
.
\end{equation*}%
Since $\psi =\mathcal{F}^{-1}\varphi \in \mathcal{S}(\mathbf{R}^{d})$, we
have $K_{1}(x)=K\ast \psi (x)$. Therefore, by (\ref{fnew}), for any
multiindex $\gamma \in \mathbf{N}_{0}^{d},\kappa \in (0,2),$%
\begin{eqnarray*}
\sup_{x}|\partial ^{\kappa }D^{\gamma }K_{1}(x)| &\leq &\sup_{x}|\partial
^{\kappa }D^{\gamma }\psi (x)|<\infty , \\
\int |\partial ^{\kappa }D^{\gamma }K_{1}(x)|dx &\leq &\int |\partial
^{\kappa }D^{\gamma }\psi (x)|dx<\infty .
\end{eqnarray*}%
By Parseval's equality and (\ref{fn2}), for any multiindices $\gamma ,\mu
,,\kappa \in (0,2),$%
\begin{eqnarray*}
\int |\partial ^{\kappa }D^{\gamma }K_{2}(x)|^{2}dx &=&\int |\left( i\xi
\right) ^{\gamma }|\xi |^{\kappa }\left[ 1-\varphi (\xi )\right] e^{-|\xi
|^{\alpha }}|^{2}d\xi <\infty , \\
\int |(ix)^{\mu }\partial ^{\kappa }D^{\gamma }K_{2}(x)|^{2}dx &=&\int
|D^{\mu }\left( (i\xi )^{\gamma }|\xi |^{\kappa }e^{-|\xi |^{\alpha
}}\right) |^{2}\left[ 1-\varphi (\xi )\right] d\xi <\infty .
\end{eqnarray*}%
Therefore, by Cauchy-Schwartz inequality with $d_{1}=[\frac{d}{4}]+1$,%
\begin{eqnarray*}
\int |\partial ^{\kappa }D^{\gamma }K_{2}(x)|dx &\leq &\left\{ \int
(1+|x|^{2})^{2d_{1}}|D^{\gamma }K_{2}(x)|^{2}dx\right\} ^{1/2}\times \\
&&\times \left\{ \int (1+|x|^{2})^{-2d_{1}}dx\right\} ^{1/2}.
\end{eqnarray*}

(ii) Changing the variable of integration in (\ref{fn1}) we get $%
G_{t}(x)=t^{-d/\alpha }K(x/t^{1/\alpha }),x\in \mathbf{R}^{d},t>0$. For any $%
v\in \mathcal{S}(\mathbf{R}^{d}),\gamma \in \mathbf{N}_{0}^{d},\kappa \in
(0,2),$%
\begin{eqnarray*}
\partial ^{\kappa }D^{\gamma }G_{t}\ast v(x) &=&\int \partial ^{\kappa
}D^{\gamma }G_{t}(x-y)v(y)dy \\
&=&t^{-d/\alpha -(|\gamma |+\kappa )/\alpha }\int \partial ^{\kappa
}D^{\gamma }K((x-y)/t^{1/\alpha })v(y)dy \\
&=&t^{-d/\alpha -(|\gamma |+\kappa )/\alpha }\int \partial ^{\kappa
}D^{\gamma }K(y/t^{1/\alpha })v(x-y)dy
\end{eqnarray*}%
and the statement follows.

(iii) Since for $v\in \mathcal{S}(\mathbf{R}^{d})$%
\begin{equation*}
G_{t}\ast v-v=\int_{0}^{t}\partial ^{\alpha }G_{s}\ast vds,0\leq t,
\end{equation*}%
it follows by part (ii),%
\begin{eqnarray*}
|G_{t}\ast v-v|_{p} &\leq &\int_{0}^{t}|\partial ^{\alpha }G_{s}\ast
v|_{p}ds=\int_{0}^{t}|\partial ^{\alpha (1-\kappa )}G_{s}\ast \partial
^{\alpha \kappa }v|_{p}ds \\
&\leq &\int_{0}^{t}s^{\kappa -1}ds|\partial ^{\alpha \kappa }v|_{p}\leq
Ct^{\kappa }|\partial ^{\alpha \kappa }v|_{p}.
\end{eqnarray*}
\end{proof}

We will need the following embedding estimate as well.

\begin{lemma}
\label{lem74}(see Lemma 7.4 in \cite{kry1}) Let $\mu \in (0,1),\mu p>1,p\geq
1,\kappa \in (0,1]$. Let $h(t)$ be a continuous $H_{p}^{\alpha \kappa }(%
\mathbf{R}^{d})$-valued function. Then there is a constant $C=C(d,\mu )$
such that for $s\leq t,$%
\begin{equation*}
|\partial ^{\alpha (1-\kappa )}[h(t)-h(s)]|_{p}^{p}\leq C(t-s)^{\mu
p-1}\int_{0}^{t-s}\frac{dr}{r^{1+\mu p}}\int_{s}^{t-r}|\partial ^{\alpha
(1-\kappa )}[h(\upsilon +r)-h(\upsilon )]|_{p}^{p}d\upsilon .
\end{equation*}
\end{lemma}

\begin{proposition}
\label{prope1}Assume $p>2,f\in \mathfrak{D}_{p}(E)$, and 
\begin{equation*}
u(t)=\int_{0}^{t}G_{s}\ast f(s)ds,0\leq t\leq T.
\end{equation*}

Let $1-\frac{1}{p}>\kappa \geq \frac{1}{2}$ (note $\frac{1}{p}<1-\kappa \leq 
\frac{1}{2}$). Then there is a constant $C$ such that for all $0\leq s\leq
t\leq T,$%
\begin{equation*}
|\partial ^{\alpha (1-\kappa )}[u(t)-u(s)]|_{p}^{p}\leq C(t-s)^{\kappa
p-1}[|f|_{p}^{p}+|\partial ^{\alpha }u|_{p}^{p}].
\end{equation*}
\end{proposition}

\begin{proof}
We apply Lemma \ref{lem74} to $u(t)=\int_{0}^{t}G_{t-s}\ast f(s)ds,0\leq
t\leq T.$ Since $G_{t+s}=G_{t}\ast G_{s}$, it follows for $v,r\geq 0,$ 
\begin{eqnarray*}
u(\upsilon +r)-u(\upsilon ) &=&\int_{0}^{\upsilon +r}G_{\upsilon +r-\tau
}\ast f(\tau )d\tau -\int_{0}^{\upsilon }G_{\upsilon -\tau }\ast f(\tau
)d\tau \\
&=&\int_{\upsilon }^{\upsilon +r}G_{\upsilon +r-\tau }\ast f(\tau )d\tau \\
&&+\int_{0}^{\upsilon }G_{\upsilon +r-\tau }\ast f(\tau )d\tau
-\int_{0}^{\upsilon }G_{\upsilon -\tau }\ast f(\tau )d\tau \\
&=&\int_{0}^{r}G_{r-\tau }\ast f(\upsilon +\tau )d\tau +G_{r}\ast u(\upsilon
)-u(\upsilon ).
\end{eqnarray*}

By H\"{o}lder inequality and Lemma \ref{lemn1} for $r>0,$%
\begin{eqnarray*}
&&|\partial ^{\alpha (1-\kappa )}\int_{0}^{r}G_{r-\tau }\ast f(\upsilon
+\tau )d\tau |_{p}^{p} \\
&=&|\int_{0}^{r}\upsilon ^{\kappa }\upsilon ^{-\kappa }\partial ^{\alpha
(1-\kappa )}G_{\tau }\ast f(\upsilon +r-\tau )d\tau |_{p}^{p} \\
&\leq &\left( \int_{0}^{r}\upsilon ^{-\kappa q}d\upsilon \right)
^{p/q}\int_{0}^{r}\upsilon ^{\kappa p}|\partial ^{\alpha (1-\kappa )}G_{\tau
}\ast f(\upsilon +r-\tau )|_{p}^{p}d\tau \\
&\leq &r^{(1-\kappa )p-1}\int_{0}^{r}\upsilon ^{(2\kappa -1)p}|f(\upsilon
+r-\tau )|_{p}^{p}d\tau \\
&\leq &r^{(1-\kappa )p-1}r^{(2\kappa -1)p}\int_{0}^{r}|f(\upsilon +r-\tau
)|_{p}^{p}d\tau \\
&=&r^{\kappa p-1}\int_{0}^{r}|f(\upsilon +r-\tau )|_{p}^{p}d\tau .\text{ }
\end{eqnarray*}%
By Lemma \ref{lemn1}, for $r,\upsilon >0,$%
\begin{equation*}
|\partial ^{\alpha (1-\kappa )}[G_{r}\ast u(\upsilon )-u(\upsilon
)]|_{p}^{p}\leq Cr^{p\kappa }|\QTR{sl}{\partial }^{\alpha }u(\upsilon
)|_{p}^{p}.
\end{equation*}%
Therefore for a fixed $\mu \in (0,1/2)$,%
\begin{eqnarray*}
&&\int_{0}^{t-s}\frac{dr}{r^{1+\mu p}}\int_{s}^{t-r}|\partial ^{\alpha
(1-\kappa )}[u(\upsilon +r)-u(\upsilon )]|_{p}^{p}d\upsilon . \\
&\leq &C[\int_{0}^{t-s}\frac{dr}{r^{1+\mu p}}\int_{s}^{t-r}r^{\kappa
p-1}\int_{0}^{r}|f(\upsilon +r-\tau )|_{p}^{p}d\tau d\upsilon \\
&&+\int_{0}^{t-s}\frac{dr}{r^{1+\mu p}}\int_{s}^{t-r}r^{\kappa p}|\partial
^{\alpha }u(\upsilon )|_{p}^{p}d\upsilon ] \\
&\leq &C[\int_{0}^{t-s}\frac{dr}{r^{2+(\mu -\kappa )p}}\int_{0}^{r}\int_{s+%
\tau }^{t-r+\tau }|f(\upsilon )|_{p}^{p}d\upsilon d\tau \\
&&+\int_{0}^{t-s}\frac{dr}{r^{1+(\mu -\kappa )p}}\int_{s}^{t-r}|\partial
^{\alpha }u(\upsilon )|_{p}^{p}d\upsilon ] \\
&\leq &C[|f|_{p}^{p}+|\partial ^{\alpha }u|_{p}^{p}](t-s)^{(\kappa -\mu )p},
\end{eqnarray*}%
and by Lemma \ref{lem74}, applied for $\mu \in (0,1/2),$%
\begin{equation*}
|\partial ^{\alpha (1-\kappa )}[u(t)-u(s)]|_{p}^{p}\leq C(t-s)^{\kappa
p-1}[|f|_{p}^{p}+|\partial ^{\alpha }u|_{p}^{p}].
\end{equation*}
\end{proof}

\begin{corollary}
\label{cor7}Let $u\in \mathcal{H}_{p}^{\alpha },p>2,p>2d/\alpha ,\beta =%
\frac{\alpha }{2}-\frac{d}{p}$. Then there is a H\"{o}lder continuous
modification of $u$ on $E$ and a constant $C$ independent of $u$ such that%
\begin{equation*}
\sup_{s,x}|u(s,x)|+\sup_{s,x\neq x^{\prime }}\frac{|u(s,x)-u(s,x^{\prime })|%
}{|x-x^{\prime }|^{\beta }}\leq C||u||_{\alpha ,p}.
\end{equation*}
\end{corollary}

\begin{proof}
By Proposition \ref{prope1} with $\kappa =1/2$ and Remark \ref{rem7}, $u$ is
H\"{o}lder continuous and%
\begin{equation*}
\sup_{0\leq s\leq T}|u(s,\cdot )|_{\alpha /2,p}\leq C||u||_{\alpha ,p}.
\end{equation*}%
By Sobolev embedding theorem, there is a constant $C$ such that%
\begin{eqnarray*}
&&\sup_{0\leq s\leq T}|u(s,x)|+\sup_{s,x\neq x^{\prime }}\frac{%
|u(s,x)-u(s,x^{\prime })|}{|x-x^{\prime }|^{\beta }} \\
&\leq &\sup_{0\leq s\leq T}|u(s,\cdot )|_{\alpha /2,p}\leq C||u||_{\alpha
,p}.
\end{eqnarray*}
\end{proof}

\section{Martingale problem}

In this section, we consider the martingale problem associated with the
operator%
\begin{equation*}
L=A+B\text{.}
\end{equation*}

Let $D=D([0,T],\mathbf{R}^{d})$ be the Skorokhod space of cadlag $\mathbf{R}%
^{d}$-valued trajectories and let $X_{t}=X_{t}(w)=w_{t},w\in D,$ be the
canonical process on it.

Let%
\begin{equation*}
\mathcal{D}_{t}=\sigma (X_{s},s\leq t),\mathcal{D}=\vee _{t}\mathcal{D}_{t},%
\mathbb{D=}\left( \mathcal{D}_{t+}\right) ,t\in \lbrack 0,T].
\end{equation*}%
We say that a probability measure $\mathbf{P}$ on $\left( D,\mathcal{D}%
\right) $ is a solution to the $(s,x,L)$-martingale problem (see \cite{st}, 
\cite{MiP923}) if $\mathbf{P}(X_{r}=x,0\leq r\leq s)=1$ and for all $v\in
C_{0}^{\infty }(\mathbf{R}^{d})$ the process%
\begin{equation}
M_{t}(v)=v(X_{t})-\int_{s}^{t}Lv(r,X_{r})]dr  \label{mart1}
\end{equation}%
is a $(\mathbb{D},\mathbf{P})$-martingale. We denote $S(s,x,L)$ the set of
all solutions to the problem $(s,x,L)$-martingale problem.

A modification of Theorem 5 in \cite{MiP923} is the following statement.

\begin{proposition}
\label{prop3}Let Assumptions A and B(i)-(ii) hold. Then for each $(s,x)\in E$
there is a unique solution $\mathbf{P}_{s,x}$ to the martingale problem $%
(s,x,L),$ and the process $\left( X_{t},\mathbb{D},(\mathbf{P}_{s,x})\right) 
$ is strong Markov.

If, in addition,%
\begin{equation}
\lim_{l\rightarrow \infty }\int_{0}^{T}\sup_{x}\pi (t,x,\left\{ |\upsilon
|>l\right\} )dt=0,  \label{mart3}
\end{equation}%
then the function $\mathbf{P}_{s,x}$ is weakly continuous in $(s,x)$.
\end{proposition}

\subsection{Auxiliary results}

We will need the following $L_{p}$-estimate.

\begin{lemma}
\label{lemn2}(cf. Lemma 3.6 in \cite{MiP923}) Let Assumptions A and
B(i)-(ii) hold. Let $p>\frac{d}{\beta }\vee \frac{2d}{\alpha }\vee 2,$ $%
(s_{0},x_{0})\in E,\mathbf{P}\in S(s_{0},x_{0},L)$.

Then there is a constant $C=C(R,T,K,\eta ,\beta ,w,p)$ such that for any $%
f\in C_{0}^{\infty }(E),$%
\begin{equation*}
\mathbf{P}\int_{s_{0}}^{\tau }f(r,X_{r})dr\leq C|f|_{p}.
\end{equation*}
\end{lemma}

\begin{proof}
Let $\zeta \in C_{0}^{\infty }(\mathbf{R}^{d}),\zeta \geq 0,\zeta (x)=\zeta
(|x|),\zeta (x)=0$ if $|x|\geq 1$, and $\int \zeta ^{p}dx=1$. For $\delta >0$
denote $\zeta _{\delta }(x)=\varepsilon ^{-d/p}\zeta (x/\delta ),x\in 
\mathbf{R}^{d}.$ Let%
\begin{eqnarray*}
u_{\delta }(t,x) &=&\int u(t,x-y)\zeta _{\delta }^{p}(y)dy \\
&=&\int u(t,y)\zeta _{\delta }^{p}(x-y)dy,(t,x)\in E.
\end{eqnarray*}%
Let 
\begin{equation*}
\tilde{L}v=Av+B^{\varepsilon _{0}}v,
\end{equation*}%
where $B^{\varepsilon _{0}}$ is defined by (\ref{f0}) with $\varepsilon _{0}$
so that the assumptions of Theorem \ref{main2} hold. Then%
\begin{equation*}
Lv=\tilde{L}v+Rv
\end{equation*}%
with 
\begin{equation*}
Rv(t,x)=\int_{|y|>\varepsilon _{0}}[v(x+y)-v(x)]\pi (t,x,dy).
\end{equation*}%
Define%
\begin{equation*}
\tilde{L}^{\delta }v=Av+B^{\varepsilon _{0},\delta }v,
\end{equation*}%
where%
\begin{equation*}
B^{\varepsilon _{0},\delta }v(t,x)=\frac{\mathbf{P[}\zeta _{\delta
}^{p}(X_{t}-x)B_{t,X_{t}}^{\varepsilon _{0}}v(x)]}{\mathbf{P}\zeta _{\delta
}^{p}(X_{t}-x)}
\end{equation*}%
(here we assume $\frac{0}{0}=0$). Since for $\tilde{L}^{\delta }$ the
assumptions of Theorem \ref{main2} hold uniformly in $\delta $, there is $%
u=u^{\delta }\in \mathcal{H}_{p}^{\alpha }(E)$ solving%
\begin{eqnarray*}
\partial _{t}u(t,x)+\tilde{L}^{\delta }u(t,x) &=&f(t,x),(t,x)\in E, \\
u(T,x) &=&0,x\in \mathbf{R}^{d}.
\end{eqnarray*}%
Moreover, there is a constant $C$ independent of $\delta $ such that 
\begin{equation}
||u^{\delta }||_{\alpha ,p}\leq C|f|_{p}.  \label{form0}
\end{equation}%
In addition, by Corollary \ref{cor7} and (\ref{form0}), there is a constant
independent of $\delta $ such that%
\begin{equation}
\sup_{s,x}|u^{\delta }(s,x)|\leq C|f|_{p}.  \label{form00}
\end{equation}%
Applying Ito formula to $u_{\delta }^{\delta }(t,x)=\int \zeta _{\delta
}^{p}(x-z)u^{\delta }(t,z)dz=\int \zeta _{\delta }^{p}(z)u^{\delta
}(t,x-z)dz,$ we have%
\begin{eqnarray}
-u_{\delta }^{\delta }(s_{0},x_{0}) &=&\int_{s_{0}}^{T}[\partial
_{t}u^{\delta }(r,z)+(A+B^{\varepsilon _{0},\delta })u^{\delta }(r,z)]\kappa
_{\delta }(t,z)dz  \notag \\
&&+\mathbf{P}\int_{s_{0}}^{T}Ru_{\delta }^{\delta
}(r,X_{r})dr+\int_{s_{0}}^{T}R_{2}(r)dr  \label{form2} \\
&=&\int_{s_{0}}^{T}\mathbf{P}f_{\delta }(r,X_{r})dr+\mathbf{P}%
\int_{s_{0}}^{T}Ru_{\delta }^{\delta }(r,X_{r})dr+\int_{s_{0}}^{T}R_{2}(r)dr,
\notag
\end{eqnarray}%
where $\kappa _{\delta }(t,z)=\mathbf{P}\zeta _{\delta }^{p}(X_{t}-z)$ and 
\begin{eqnarray*}
R_{2}(r) &=&\int \mathbf{P[}\zeta _{\delta
}^{p}(X_{r}-z)A_{r,X_{r}}u^{\delta }(r,z)-\zeta _{\delta
}^{p}(X_{r}-z)A_{r,z}u^{\delta }(r,z)]dz \\
&=&\mathbf{P}\int \int \nabla _{y}^{\alpha }u^{\delta
}(r,z)[m(r,X_{r},y)-m(r,z,y)]\frac{dy}{|y|^{d+\alpha }}\zeta _{\delta
}^{p}(X_{r}-z)dz,
\end{eqnarray*}

$s_{0}\leq r\leq T$. By H\"{o}lder inequality, for any $r\in \lbrack
s_{0},T],$%
\begin{equation*}
|R_{2}(r)|^{p}\leq \mathbf{P}\int \left\vert \int \nabla _{y}^{\alpha
}u^{\delta }(r,z)[m(r,X_{r},y)-m(r,z,y)]\zeta _{\delta }(X_{r}-z)\frac{dy}{%
|y|^{d+\alpha }}\right\vert ^{p}dz.
\end{equation*}%
We will show that%
\begin{equation}
\int_{s_{0}}^{T}|R_{2}(r)|^{p}dr\rightarrow 0\text{ as }\delta \rightarrow 0.
\label{form1}
\end{equation}%
According to Lemma \ref{rem5},%
\begin{equation*}
|R_{2}(r)|^{p}\leq C|\partial ^{\alpha }u^{\delta }(r)|_{p}^{p}\mathbf{P}%
\int \sup_{y}[|M(r,z,y)|^{p}+|\partial ^{\beta }M(r,z,y)|^{p}]dz,
\end{equation*}%
where%
\begin{equation*}
M(r,z,y)=[m(r,X_{r},y)-m(r,z,y)]\zeta _{\delta }(X_{r}-z).
\end{equation*}%
Obviously,%
\begin{equation*}
\left\vert m(r,X_{r},y)-m(r,z,y)\right\vert \zeta _{\delta }(X_{r}-z)\leq
w(\delta )\zeta _{\delta }(X_{r}-z)
\end{equation*}%
and%
\begin{equation*}
\int \sup_{y}[|M(r,z,y)|^{p}dz\leq \int w(\delta )^{p}\zeta _{\delta
}^{p}(X_{r}-z)dz=w(\delta )^{p}.
\end{equation*}%
Denoting $m_{0}(r,z,y)=m(r,X_{r},y)-m(r,z,y)$, we have%
\begin{eqnarray*}
&&|\partial ^{\beta }\left( m_{0}(r,z,y)\zeta _{\delta }(z-X_{r})\right) | \\
&\leq &\int_{|\upsilon |>\delta }\left\vert m_{0}(t,z+\upsilon ,y)\zeta
_{\delta }(z+v-X_{r})-m_{0}(t,z,y)\zeta _{\delta }(z-X_{r})\right\vert \frac{%
d\upsilon }{|\upsilon |^{1+\beta }} \\
&&+\zeta _{\delta }(z-X_{r})\int_{|\upsilon |\leq \delta }|m(r,z+\upsilon
,y)-m(r,z,y)|\frac{d\upsilon }{|\upsilon |^{d+\beta }} \\
&&+|m_{0}(t,z,y)|\int_{|\upsilon |\leq \delta }\left\vert \zeta _{\delta
}(z+\upsilon -X_{r})-\zeta _{\delta }(z-X_{r})\right\vert \frac{d\upsilon }{%
|\upsilon |^{d+\beta }} \\
&&+\int_{|\upsilon |\leq \delta }\left\vert \zeta _{\delta }(z+\upsilon
-X_{r})-\zeta _{\delta }(z-X_{r})\right\vert \left\vert m(r,z+\upsilon
,y)-m(r,z,y)\right\vert \frac{d\upsilon }{|\upsilon |^{d+\beta }} \\
&=&(H_{1}+H_{2}+H_{3}+H_{4})(r,z,y),(r,z,y)\in \lbrack 0,T]\times \mathbf{R}%
^{d}\times \mathbf{R}^{d}.
\end{eqnarray*}%
Obviously, 
\begin{equation*}
\int \sup_{y}H_{2}(r,z,y)^{p}dz\leq (\int_{|\upsilon |\leq \delta
}w(\upsilon )\frac{d\upsilon }{|\upsilon |^{d+\beta }})^{p}.
\end{equation*}%
It follows, by H\"{o}lder inequality,%
\begin{eqnarray*}
\int \sup_{y}H_{1}(r,z,y)^{p}dz &\leq &Cw(\delta )^{p}\delta ^{-p\beta }, \\
\int \sup_{y}H_{4}(r,z,y)^{p}dz &\leq &C(\int_{|\upsilon |\leq \delta
}w(\upsilon )\frac{d\upsilon }{|\upsilon |^{d+\beta }})^{p}.
\end{eqnarray*}%
Changing the variable of integration,%
\begin{eqnarray*}
H_{3}(r,z,y) &\leq &w(2\delta )\int_{|\upsilon |\leq \delta }\left\vert
\zeta _{\delta }(z+\upsilon -X_{r})-\zeta _{\delta }(z-X_{r})\right\vert 
\frac{d\upsilon }{|\upsilon |^{d+\beta }} \\
&=&w(2\delta )\delta ^{-\beta }\delta ^{-d/p}\int_{|\upsilon |\leq
1}\left\vert \zeta (\frac{z-X_{r}}{\delta }+\upsilon )-\zeta (\frac{z-X_{r}}{%
\delta })\right\vert \frac{d\upsilon }{|\upsilon |^{d+\beta }}
\end{eqnarray*}%
and%
\begin{eqnarray*}
&&\int \sup_{y}H_{3}(r,z,y)^{p}dz \\
&\leq &w(2\delta )^{p}\delta ^{-\beta p}\int (\int_{|\upsilon |\leq
1}\left\vert \zeta (z+\upsilon )-\zeta (z)\right\vert \frac{d\upsilon }{%
|\upsilon |^{d+\beta }})^{p}dz \\
&\leq &Cw(2\delta )^{p}\delta ^{-\beta p}.
\end{eqnarray*}%
Therefore by (\ref{form0}), (\ref{form1}) follows. Since 
\begin{equation*}
\int_{s_{0}}^{T}\mathbf{P}f_{\delta }(r,X_{r})dr=-\left( u_{\delta }^{\delta
}(s_{0},x_{0})+\mathbf{P}\int_{s_{0}}^{T}Ru_{\delta }^{\delta
}(r,X_{r})dr+\int_{s_{0}}^{T}R_{2}(r)dr\right)
\end{equation*}%
(see (\ref{form2})), the statement follows by (\ref{form00}) and (\ref{form1}%
) passing to the limit as $\delta \rightarrow 0.$
\end{proof}

\begin{corollary}
\label{cor8}Let Assumptions A and B hold, $(s_{0},x_{0})\in E$. Then the set 
$S(s_{0},x_{0},L)$ consists of at most one probability measure.
\end{corollary}

\begin{proof}
Let $f\in C_{0}^{\infty }(E),p>\frac{d}{\beta }\vee \frac{2d}{\alpha }\vee 2$%
. By Theorem \ref{main3}, there is $u\in \mathcal{H}_{p}^{\alpha }(E)$
solving 
\begin{eqnarray*}
\partial _{t}u(t,x)+Lu(t,x) &=&f(t,x),(t,x)\in E, \\
u(T,x) &=&0,x\in \mathbf{R}^{d}.
\end{eqnarray*}%
Let $\varphi \in C_{0}^{\infty }(\mathbf{R}^{d}),\varphi \geq 0,\int \varphi
dx=1,\varphi _{\varepsilon }(x)=\varepsilon ^{-d}\varphi (x/\varepsilon
),x\in \mathbf{R}^{d},$ and%
\begin{equation*}
u_{\varepsilon }(t,x)=\int u(t,x-y)\varphi _{\varepsilon }(y)dy,(t,x)\in 
\mathbf{R}^{d}\text{.}
\end{equation*}%
Applying Ito formula, we have%
\begin{equation*}
-u_{\varepsilon }(s_{0},x_{0})=\mathbf{P}\int_{s_{0}}^{T}[\partial
_{t}u_{\varepsilon }(r,X_{r})+Lu_{\varepsilon }(r,X_{r})]dr.
\end{equation*}%
Using Lemma \ref{lemn2} and Corollary \ref{cor7} to pass yo the limit we
derive that%
\begin{equation*}
-u(s_{0},x_{0})=\mathbf{P}\int_{s_{0}}^{T}f(r,X_{r})dr
\end{equation*}%
and the uniqueness follows by Lemma 2.4 in \cite{MiP923}.
\end{proof}

Now we can construct a "local" solution of the martingale problem.

\begin{lemma}
\label{lemn3}Let Assumptions A, B(i)-(ii) hold, $\pi (t,x,d\upsilon )=\chi
_{\left\{ |x|\leq R\right\} }\pi (t,x,d\upsilon ),(t,x)\in E,$ for some $%
R>0. $

Then for each $(s,x)\in E$ there is a unique solution $\mathbf{P}_{s,x}\in
S(s,x,L)$ and $\mathbf{P}_{s,x}$ is weakly continuous in $(s,x)$.
\end{lemma}

\begin{proof}
Let $\varphi \in C_{0}^{\infty }(\mathbf{R}^{d}),\varphi \geq 0,\int \varphi
dx=1,\varphi _{\varepsilon }(x)=\varepsilon ^{-d}\varphi (x/\varepsilon
),x\in \mathbf{R}^{d},$ and%
\begin{equation*}
\pi _{\varepsilon }(t,x,d\upsilon )=\int \pi (t,x-z,d\upsilon )\varphi
_{\varepsilon }(z)dz,(t,x)\in E.
\end{equation*}%
Let $\varepsilon _{n}\rightarrow 0$ and let $L^{n}$ be an operator defined
as $L$ with $\pi $ replaced by $\pi _{\varepsilon _{n}}$. It follows by
Theorem IX.2.31 in \cite{jaks} that the set $\mathcal{S}(s,x,L^{n})\neq
\emptyset .$ Since by Lemma \ref{lemn2}, for $\mathbf{P}_{s,x}^{n}\in
S(s,x,L^{n})$, 
\begin{eqnarray*}
&&\mathbf{P}_{s,x}^{n}\int_{s}^{T}\pi _{\varepsilon _{n}}(r,X_{r},\left\{
|\upsilon |>l\right\} )dr \\
&\leq &C\int_{s}^{T}\int \pi (r,x,\left\{ |\upsilon |>l\right\}
)dxdr\rightarrow 0\text{ as }l\rightarrow \infty \text{,}
\end{eqnarray*}%
the sequence $\left\{ \mathbf{P}_{s,x}^{n}\right\} $ is tight (see Theorem
VI.4.18 in \cite{jaks})$.$ Obviously, for each $v\in C_{0}^{\infty }(\mathbf{%
R}^{d})$, $L^{n}v(t,x)\rightarrow Lv(t,x)$ $dtdx$-a.e. Therefore, by Lemma
3.7 in \cite{MiP923} the set $S(s,x,L)\neq \emptyset .$ By Lemma \ref{lemn2}%
, the solution $\mathbf{P}_{s,x}\in S(s,x,L)$ is unique. Applying Lemma 3.7
in \cite{MiP923} again, we see that $\mathbf{P}_{s,x}$ is continuous in $%
(s,x)$.
\end{proof}

\begin{corollary}
\label{cor9}Let Assumptions A, B(i)-(ii) hold. Then for each $(s,x)\in E$,
there is at most one solution $\mathbf{P}_{s,x}\in S(s,x,L)$.
\end{corollary}

\begin{proof}
The statement is immediate consequence of Lemma \ref{lemn3} and Theorem
1.6(b) in \cite{MiP923}$.$
\end{proof}

\subsection{Proof of Proposition \protect\ref{prop3}}

The uniqueness follows by Corollary \ref{cor9}. In the first part of the
proof we assume that (\ref{mart3}) holds and use weak convergence arguments.
In the second part, we cover the general case by putting together measurable
families of probability measures.

(i) Assume (\ref{mart3}) holds. Let $L^{n}$ be an operator defined as $L$
with $\pi $ replaced by $\chi _{\left\{ |x|\leq n\right\} }\pi $. According
to Lemma \ref{lemn3}, for each $(s,x)\in E$ there is a unique and $\mathbf{P}%
_{s,x}^{n}\in S(s,x,L^{n})$ and $\mathbf{P}_{s,x}^{n}$ is weakly continuous
in $(s,x)$. By Theorem VI.4.18 in \cite{jaks}, $\left\{ \mathbf{P}%
_{s,x}^{n}\right\} $ is tight $.$Since $L^{n}v\rightarrow Lv$ $dxdt$-a.e.
and $L^{n}v$ is uniformly bounded for any $v\in C_{0}^{\infty }(\mathbf{R}%
^{d})$, by Lemma 3.7 in \cite{MiP923}$,\,\ $the sequence $\mathbf{P}%
_{s,x}^{n}\rightarrow \mathbf{P}_{s,x}\in S(s,x,L)$ weakly $(\mathbf{P}%
_{s,x} $ is unique by Corollary \ref{cor9}). The same Lemma 3.7, \cite%
{MiP923}, implies that $\mathbf{P}_{s,x}$ is weakly continuous in $(s,x)$.

(ii) In the general case (without assuming (\ref{mart3})), we split the
operator $Lu=\tilde{L}u+\tilde{B}u$, where $\tilde{L}$ is defined as $L$
with $\pi (t,x,d\upsilon )$ replaced by $\chi _{\left\{ |\upsilon
|<1\right\} }\pi (t,x,d\upsilon )$, and%
\begin{equation*}
\tilde{B}_{t,x}u(x)=\int_{|\upsilon |\geq 1}[u(x+\upsilon )-u(x)]\pi
(t,x,d\upsilon ),(t,x)\in E,u\in C_{0}^{\infty }(\mathbf{R}^{d}).
\end{equation*}%
Let $\left( \Omega _{2},\mathcal{F}_{2},\mathbf{P}_{2}\right) $ be a
probability space with a Poisson point measure $\tilde{p}(dt,dz)$ on $%
[0,\infty )\times (\mathbf{R\backslash \{}0\mathbf{\})}$ with 
\begin{equation*}
\mathbf{E}\tilde{p}(dt,dz)=\frac{dzdt}{|z|^{2}}.
\end{equation*}%
According to Lemma 14.50 in \cite{Jac79}, there is a measurable $\mathbf{R}%
^{d}\cap \left\{ |\upsilon |\geq 1\right\} $-valued function $c(t,x,z)$ such
that for any Borel $\Gamma $%
\begin{equation*}
\int_{\Gamma }\chi _{\left\{ |\upsilon |\geq 1\right\} }\pi (t,x,d\upsilon
)=\int \chi _{\Gamma }(c(t,x,z))\frac{dz}{z^{2}},(t,x)\in E.
\end{equation*}%
Consider the probability space%
\begin{equation*}
\left( \Omega ,\mathcal{F},\mathbf{P}_{s,x}^{\prime }\right) =\left( \Omega
_{2}\times D,\mathcal{F}_{2}\otimes \mathcal{D},\mathbf{P}_{2}\otimes 
\mathbf{P}_{s,x}\right) .
\end{equation*}%
Let%
\begin{eqnarray*}
H_{t} &=&\int_{s\wedge t}^{t}\int c(r,X_{r-,}z)\tilde{p}(dr,dz),s\leq t\leq
T, \\
\tau &=&\inf (t>s:\Delta H_{t}=H_{t}-H_{t-}\neq 0)\wedge T, \\
K_{t} &=&\chi _{\{\tau \leq t\}}, \\
Y_{t} &=&X_{t\wedge \tau }+H_{t\wedge \tau },0\leq t\leq T.
\end{eqnarray*}%
Note that $\tau =\inf (t>s:\Delta H_{t}\neq 0)\wedge T=\tau =\inf
(t>s:|\Delta H_{t}|\geq 1)\wedge T$. Let $\hat{D}=D([0,T],\mathbf{R}%
^{d}\times \lbrack 0,\infty ))$ be the Skorokhod space of cadlag $\mathbf{R}%
^{d}\times \lbrack 0,\infty )$-valued trajectories and let $%
Z_{t}=Z_{t}(w)=(y_{t}(w),k_{t}(w))=w_{t}\in R^{d}\times \lbrack 0,\infty
),w\in \hat{D}$ be the canonical process on it. Let%
\begin{equation*}
\mathcal{\hat{D}}_{t}=\sigma (Z_{s},s\leq t),\mathcal{\hat{D}}=\vee _{t}%
\mathcal{\hat{D}}_{t},\mathbb{\hat{D}=}\left( \mathcal{\hat{D}}_{t+}\right)
,t\in \lbrack 0,T].
\end{equation*}%
Denote $\mathbf{\hat{P}}_{s,x}^{1}$ the measure on $\hat{D}$ induced by $%
(Y_{t},K_{t}),0\leq t\leq T.$ Let 
\begin{eqnarray*}
\tau _{1} &=&\inf \left( t>s:\Delta k_{t}\geq 1\right) \wedge T,\ldots , \\
\tau _{n+1} &=&\inf \left( t>\tau _{n}:\Delta k_{t}\geq 1\right) \wedge T, \\
\mathcal{\hat{D}}_{\tau _{n}} &=&\sigma (Z_{t\wedge \tau _{n}},0\leq t\leq
T),n\geq 1.
\end{eqnarray*}%
Then $\left( \mathbf{\hat{P}}_{s,x}^{1}\right) $ is a measurable family of
measures on $\left( \hat{D},\mathcal{\hat{D}}\right) $ and for each $v\in
C_{0}^{\infty }(\mathbf{R}^{d}),$%
\begin{equation}
\hat{M}_{t\wedge \tau _{n}}(v)=v(y_{t\wedge \tau _{n}})-\int_{s}^{t\wedge
\tau _{n}}Lv(r,y_{r})dr,s\leq t\leq T,  \label{form5}
\end{equation}%
is $\left( \mathbf{\hat{P}}_{s,x}^{1},\mathbb{\hat{D}}\right) $-martingale
with $n=1$. Let us introduce the mappings%
\begin{equation*}
\mathcal{J}_{\tau _{1}}(w,w^{\prime })_{t}=\left\{ 
\begin{array}{cc}
w_{t} & \text{if }t<\tau _{1}(w), \\ 
w_{t}^{\prime } & \text{if }t\geq \tau _{1}(w),%
\end{array}%
\right.
\end{equation*}%
and let%
\begin{equation*}
Q(dw,dw^{\prime })=\mathbf{\hat{P}}_{\tau _{1}(w),X_{\tau
_{1}(w)}(w)}^{1}(dw^{\prime })\mathbf{\hat{P}}_{s,x}^{1}(dw).
\end{equation*}%
Then $\mathbf{P}_{s,x}^{2}=\mathcal{J}_{\tau _{1}}(Q)$, the image of $Q$
under $\mathcal{J}_{\tau _{1}},$ is a measurable family of measures on $\hat{%
D}$, and by Lemma 2.3 in \cite{MiP923}, $\hat{M}_{t\wedge \tau _{2}}$ is $%
\left( \mathbf{\hat{P}}_{s,x}^{2},\mathbb{\hat{D}}\right) $-martingale and $%
\mathbf{\hat{P}}_{s,x}^{2}|_{\mathcal{\hat{D}}_{\tau _{1}}}=\mathbf{\hat{P}}%
_{s,x}^{1}|_{\mathcal{\hat{D}}_{\tau _{1}}}$. Continuing and using Lemma 2.3
in \cite{MiP923}, we construct a sequence of measures $\mathbf{\hat{P}}%
_{s,x}^{n}$ such that 
\begin{equation*}
\mathbf{\hat{P}}_{s,x}^{n+1}|_{\mathcal{\hat{D}}_{\tau _{n}}}=\mathbf{\hat{P}%
}_{s,x}^{n}|_{\mathcal{\hat{D}}_{\tau _{n}}}
\end{equation*}%
and $\hat{M}_{t\wedge \tau _{n}}$ is $\left( \mathbf{\hat{P}}_{s,x}^{n},%
\mathbb{\hat{D}}\right) $-martingale. Since%
\begin{eqnarray*}
\mathbf{\hat{P}}_{s,x}^{n}\left( \tau _{n}<T\right) &=&\mathbf{\hat{P}}%
_{s,x}^{n}\left( k_{T\wedge \tau _{n}}\geq n\right) \leq
n^{-1}\int_{s}^{T\wedge \tau _{n}}\pi (r,y_{r},\left\{ |\upsilon |\geq
1\right\} )dr \\
&\leq &n^{-1}KT\rightarrow 0\text{ as }n\rightarrow \infty ,
\end{eqnarray*}%
there is a measurable family $\left( \mathbf{\hat{P}}_{s,x}\right) $ on $%
\hat{D}$ such that%
\begin{equation*}
\mathbf{\hat{P}}_{s,x}|_{\mathcal{\hat{D}}_{\tau _{n}}}=\mathbf{\hat{P}}%
_{s,x}^{n}|_{\mathcal{\hat{D}}_{\tau _{n}}},n\geq 1,
\end{equation*}%
and $\hat{M}_{t\wedge \tau _{n}}$ is $\left( \mathbf{\hat{P}}_{s,x},\mathbb{%
\hat{D}}\right) $-martingale for every $n$. Obviously, $y_{\cdot }$ under $%
\mathbf{\hat{P}}_{s,x}$ gives a measurable family $\mathbf{P}_{s,x}\in
S(s,x,L)$. The strong Markov property is a consequence of Lemma 2.2 in \cite%
{MiP923}. The statement of Proposition \ref{prop3} follows.


\begin{thebibliography}{99}
\bibitem{AbK09} Abels, H. and Kassman, M., The Cauchy problem and the
martingale problem for integro-differential operators with non-smooth
kernels, Osaka J. Math. 46 (2009) 661-683.

\bibitem{cav1} Caffarelli, L., Vasseur, A., Drift diffusion equations with
fractional diffusion and the quasigeostrophic equation, Annals of Math.,
Vol. 171, No. 3, 2010, 1903-1930.

\bibitem{GiT83} Gilbarg, D. and Trudinger, N. S. \emph{Elliptic Partial
Differential Equations of Second Order}. Springer, New York, 1983.

\bibitem{KimDong} Dong H. and Kim D., On Lp estimates of non-local elliptic
equations, arXiv:1102.4073v1 [math.AP], 2011.

\bibitem{Jac79} Jacod, J., \textit{Calcul Stochastique et Probl\`{e}mes de
Martingales}, \textit{Lecture Notes in Mathematics}, 714, Springer Verlag,
Berlin New York, 1979.

\bibitem{jaks} Jacod, J. and Shiryaev, A.N., Limit Theorems for Stochastic
Processes, Springer, 1987.

\bibitem{Kom84} Komatsu, T., On the Martingale Problem for Generators of
Stable Processes with Perturbations, \textit{Osaka J. of Math.} 22-1 (1984)
113-132.

\bibitem{kry1} Krylov, N.V., An analytic approach to SPDEs, In: \textit{%
Stochastic Partial Differential Equations: Six Perspectives}, AMS, 1999.

\bibitem{kry} Krylov, N.V., \textit{Lectures on Elliptic and Parabolic
Equations in Sobolev Spaces}, AMS, 2008.

\bibitem{MiP922} Mikulevi\v{c}ius, R. and Pragarauskas, H., On the Cauchy
Problem for Certain Integro-Differential Operators in Sobolev and H\"{o}lder
Spaces, \textit{Lithuanian Mathematical Journal} 32-2 (1992) 238-264.

\bibitem{MiP923} Mikulevi\v{c}ius, R. and Pragarauskas, H., On the
Martingale Problem Associated with Nondegenerate L\'{e}vy Operators, \textit{%
Lithuanian Mathematical Journal} 32-3 (1992) 297-311.

\bibitem{mikprag} Mikulevi\v{c}ius, R. and Pragarauskas, H., On $L_{p}$-
theory for Zakai equation with discontinuous observation process,
arXiv:1012.5816v1 [math.PR], 2010.

\bibitem{mikprag1} Mikulevi\v{c}ius, R. and Pragarauskas, H., On the Cauchy
problem for integro-differential operators in H\"{o}lder classes and the
uniqueness of the martingale problem, arXiv: 1103.3492v2 [math.AP], 2011.

\bibitem{Tri92} Triebel, H., \textit{Theory of Function Spaces II}.
Birkhaueser Verlag, 1992.

\bibitem{stein} \ Stein E. M., \textit{Singular Integrals and
Differentiability Properties of Functions}, Princeton University Press,
Princeton, 1970.

\bibitem{st} Stroock, D.W., Diffusion processes associated with Levy
generators, Z. Wahrsch. Verw. Gebiete, 32 (1975), 209-244.
\end{thebibliography}
\end{document}